\documentclass{mathincs}

\usepackage{amssymb,url}

\usepackage{amsthm,amsmath,amsbsy,ascmac,graphicx} 

\newlength{\myheight}
\newlength{\myheighta}
\DeclareMathOperator{\Mono}{Mono}

\DeclareMathOperator{\Id}{{\it Id}}

\newcommand{\Set}[1]{\left\{#1\right\}}
\newcommand{\setDef}[2]{{#1}\left|\,\vphantom{#1}{#2}\right.}
\newcommand{\SetDef}[2]{\Set{\setDef{#1}{#2}}}

\DeclareMathOperator{\C}{\mathbb{C}}
\DeclareMathOperator{\N}{\mathbb{N}}
\DeclareMathOperator{\V}{\mathbb{V}}
\DeclareMathOperator{\U}{\mathbb{U}}
\DeclareMathOperator{\Q}{\mathbb{Q}}

\DeclareMathOperator{\hht}{hm}
\DeclareMathOperator{\hc}{hc}

 \newtheorem{thm}{Theorem}[section]
 
 \newtheorem{lem}[thm]{Lemma}
 
 \theoremstyle{definition}
 \newtheorem{defn}[thm]{Definition}
 \theoremstyle{remark}
 
 \newtheorem{exmp}{Example}
 \numberwithin{equation}{section}

\usepackage{amssymb}

\begin{document}

\title[Methods for computing $b$-functions]
 {Methods for computing $b$-functions  associated \\ with $\mu$-constant deformations  \\ -- Case of inner modality 2 --}

\thanks{This work has been partly supported by JSPS  Grant-in-Aid for Scientific Research (C) (Nos 18K03214,18K03320).}

\author[Katsusuke Nabeshima]{Katsusuke Nabeshima}

\address{%
Graduate School of Technology, Industrial and Social Sciences, Tokushima University, \\2-1, Minamijosanjima, Tokushima, JAPAN}

\email{nabeshima@tokushima-u.ac.jp}

\author{Shinichi Tajima}
\address{Graduate School of Science and Technology, Niigata University, \\8050, Ikarashi 2-no-cho, Nishi-ku Niigata, JAPAN}
\email{tajima@emeritus.niigata-u.ac.jp}



\begin{abstract}
New methods for computing parametric local $b$-functions are introduced for $\mu$-constant deformations of semi-weighted homogeneous singularities. The keys of the methods are comprehensive Gr\"obner systems in Poincar\'e-Birkhoff-Witt algebra and holonomic ${\mathcal D}$-modules. It is shown that the use of semi-weighted homogeneity reduces the computational complexity of $b$-functions associated with $\mu$-constant deformations. In the case of inner modality 2, local $b$-functions associated with $\mu$-constant deformations are obtained by the resulting method and given the list of parametric local $b$-functions.

\end{abstract}


\maketitle

\section{Introduction}
We introduce a new algorithm for computing local $b$-functions of semi-weighted homogeneous polynomials, and study local $b$-functions of $\mu$-constant deformation of inner modality 2 singularities.

A $b$-function is an analytic invariant of a hypersurface singularity.  In 1997, T. Oaku gave an algorithm for computing $b$-functions \cite{Oaku1}. After that many researcher have improved the algorithm in the context of symbolic computation. 
However, as the computational complexity of the existing algorithms is still quite high, it is difficult to obtain $b$-functions in realistic time in many cases.  The bottleneck lies in computing non-commutative Gr\"obner bases. 

In a pioneering paper \cite{yano78} published in 1978, T. Yano investigated  $b$-functions and already considered $b$-functions associated with a $ \mu $-constant deformations of singularities. He noticed that some roots of $b$-functions are not stable and change these values discontinuously under  $ \mu $-constant deformations. Later, in 1980's, M. Kato and P. Cassou Nogu\'es explicitly computed $b$-functions associated with  a $ \mu$-constant deformation for some cases in their seminal papers \cite{CN1,CN2,Ka81,Ka82}.

In this paper, we consider $b$-functions of semi-weighted homogeneous singularities in the context of computational algebraic analysis. Upon using properties of semi-weighted homogeneous singularities, we construct an effective method for computing $b$-functions. Notably, the resulting algorithm explicitly compute parameter dependency of $b$-functions associated with $ \mu$-constant deformations. The keys of our approach are comprehensive Gr\"obner systems in Poincar\'e-Birkhoff-Witt algebra and holonomic ${\mathcal D}$-modules. We show that the use of semi-weighted homogeneity allows us to design an effective algorithm for computing $b$-functions associated with $ \mu$-constant deformations, which avoid the use of Gr\"obner bases computation in an elimination step. As an application, we study inner modality 2 singularities.

This paper is organized as follows. In section~2, we recall the notion of semi-weighted homogeneity and give a list of semi-weighted homogeneous polynomials whose inner modality is equal to two, the target singularities of the present paper. In section 3, we briefly recall a comprehensive Gr\"obner system, a method of  computing parametric $b$-functions and review our previous results reported in \cite{NT17a}. In section 4, we present an algorithm for computing algebraic local cohomology solutions of a holonomic ${\mathcal D}$-module. 
In section~5, we introduce a new algorithm for computing $b$-functions of semi-weighted homogeneous polynomials. In section~6, we give a list of $b$-functions associated with a $\mu$-constant deformation of inner modality 2 singularities.\\

Throughout this paper, we use the notation $x$ as the abbreviation of $n$ variables $x_1,\ldots, x_n$, $\Q$ as the field of rational numbers and $\C$ as the field of complex numbers. The set of natural numbers $\N$ includes zero. 
For elements $ p_1,p_2,\ldots,p_r $ in a ring $R$, let $ \Id(p_1,p_2,\ldots, p_r)$ denote the ideal in $R$ generated by $p_1,p_2,\ldots, p_r$.

\section{The list of inner modality 2 singularities}

Let ${\bf w}=(w_1,w_2,\ldots,w_n) \in \N^n$,  $\alpha=(a_1,a_2,\ldots,a_n) \in \N^n$ and $x^\alpha=x_1^{a_1}x_2^{a_2}\cdots x_n^{a_n} \in \C[x]$ .  Let $|x^\alpha|_{{\bf w}}$ denote the weighted degree $\displaystyle \sum_{i=1}^nw_ia_i$ of the monomial $ x^{\alpha}$. 

\begin{defn}
\begin{enumerate}
\item[(1)] A non-zero polynomial $f \in \C[x]$ is called  weighted homogeneous of type $(d; {\bf w})$ if all monomials of $f$ have the same weighted degree $d$ w.r.t. ${\bf w}$ where $d \in \N$.
\item[(2)] The polynomial $f$ is called semi-weighted homogeneous of type $(d; {\bf w})$ if $f$ is of the form $f=f_0+g$ where $f_0$ is a weighted homogeneous polynomial of type $(d; {\bf w})$ with an isolated singularity at the origin $O$ in $\C^n$, and $f=f_0$ or $\text{ord}_{{\bf w}}(f-f_0)>d$ where $\text{ord}_{{\bf w}}(f)=\text{min}\{|x^\alpha|_{{\bf w}} : x^\alpha \text{ is a monomial of  } f\}$. A monomial of $g$ is called an upper of $f_0$.
\end{enumerate}
\end{defn}

\begin{table}[ht]
\begin{center}
\caption{List of inner modality 2 singularities}~\label{modality2} 
\begin{tabular}{|c|c|c|c|} 
 \hline
type & weighted homo. $f_0$ & upper monomials & note\\ \hline
$E_{18}$ & $x^3+y^{10}$ & $xy^7, xy^8$ & \\ \hline
$E_{19}$ & $x^3+xy^7$ & $y^{11}, y^{12}$ & \\  \hline
$E_{20}$ & $x^3+y^{11}$ & $xy^{8}, xy^{9}$ & \\  \hline
$W_{17}$ & $x^4+xy^5$ & $y^{7}, y^{8}$ & \\  \hline
$W_{18}$ & $x^4+y^7$ & $x^2y^4, x^2y^5$ & \\  \hline
$Z_{17}$ & $x^3y+y^8$ & $xy^{6}, xy^{7}$ & \\  \hline
$Z_{18}$ & $x^3y+xy^6$ & $y^{9}, y^{10}$ & \\  \hline
$Z_{19}$ & $x^3y+y^9$ & $xy^{7}, xy^{8}$ & \\  \hline
$Q_{16}$ & $x^3+yz^2+y^7$ & $xy^{5}, xz^{2}$ & \\  \hline
$Q_{17}$ & $x^3+yz^2+xy^5$ & $y^{8}, y^{9}$ & \\  \hline
$Q_{18}$ & $x^3+yz^2+y^8$ & $xy^{6}, xz^{2}$ & \\  \hline
$S_{16}$ & $x^2z+yz^2+xy^4$ & $y^{6}, z^{3}$ & \\  \hline
$S_{17}$ & $x^2z+yz^2+y^6$ & $y^{4}z, z^{3}$ & \\  \hline
$U_{16}$ & $x^3+xz^2+y^5$ & $y^2z^{2}, y^3z^{2}$ & \\  \hline
$J_{16}$ & $x^3+y^9+u_1x^2y^3$ & $y^{10}$ & $4u_1^3+27\neq 0$\\  \hline
$W_{15}$ & $x^4+y^6+u_1x^2y^3$ & $y^{7}$ & $u_1^2-4\neq 0$\\  \hline
$Z_{15}$ & $x^3y+y^7+u_1x^2y^3$ & $y^{8}$ & $4u_1^3+27\neq 0$ \\  \hline
$Q_{14}$ & $x^3+yz^2+u_1x^2y^2+xy^4$ & $y^{7}$ & $u_1^2-4\neq 0$\\  \hline
$S_{14}$ & $x^2z+yz^2+y^5+u_1y^3z$ & $z^{3}$ & $u_1^2-4\neq 0$ \\  \hline
$U_{14}$ & $x^3+xz^2+u_1xy^3+y^3z$ & $yz^{3}$ & $u_1^2+1\neq 0$\\  \hline
 \end{tabular} 
\end{center}
\end{table}

Note that the Milnor number at the origin of $ f= f_0 + g $ and that of $ f_0 $ are same, and the embedded topological type of $ f=f_0 +g $ singularity and that of $ f_0 $ are same \cite{LeR76,Va}. 
Accordingly, $ f=f_0 +g $ is called a $ \mu$-constant deformation of $ f_0. $

In \cite{YS}, Yoshinaga and Suzuki gave lists of normal forms of quasihomogeneous (weighted homogeneous) functions with inner modality $\le 4$. Table~\ref{modality2} quoted from \cite{YS} is the list of inner modality~2 singularities.

For example, $E_{18}$ in the table means the following. 
\begin{enumerate}
\item[(i)] $ f_0(x,y) = x^3+y^{10} $ is a weighted homogeneous polynomial (of type $(30, (10,3))$.
\item[(ii)] The Milnor number of $ f_0(x,y) $ at the origin is equal to $ 18. $
\item[(iii)] $ f(x,y) = f_0(x,y) + c_{(1,7)}xy^7+ c_{(1,8)}xy^8 $ is a semi-weighted homogeneous polynomial (of type (30, (10,3)) where $c_{(1,7)},c_{(1,8)} \in \C$.
\end{enumerate}

\section{Comprehensive Gr\"obner system approach}

Here we briefly recall a comprehensive Gr\"obner system approach to compute parametric $b$-functions and review our previous results reported in \cite{NT17a} on the computation of $b$-functions associated with  $ \mu$-constant deformations.

Let $D=\C[x,\partial]$ denote the Weyl algebra, the ring of linear partial differential operators with coefficients in $\mathbb{C}[x]$, where $\partial = \{\partial_1,\ldots,\partial_n\}$, $\partial_i =\frac{\partial}{\partial x_i}$, i.e.,
$$\displaystyle D=\SetDef{\sum_{\beta \in \N^n}h_{\beta}(x)\partial^\beta}{h_{\beta}(x)\in \C[x]}.$$ 
Throughout the paper we assume that a linear partial differential operator is always represented in the canonical form : each power product of a partial differential operator is written as $x^\alpha{\partial}^{\beta}$ where $\alpha, \beta \in \N^n$.

Let $u=\{u_1,\ldots,u_m\}$ be variables such that $u \cap x=\emptyset$, $D[u]$ a ring of partial differential operators with coefficients in a polynomial ring $(\C[u])[x]$, i.e., $$D[u]=\SetDef{\sum_{\beta \in \N^n}h_{\beta}(u,x)\partial^\beta }{h_{\beta}(u,x)\in (\C[u])[x]}.$$ 
Let $D\langle s, \partial_t\rangle$ (or $D[u]\langle s, \partial_t\rangle$) denote the Poincar\'e-Birkhoff-Witt (PBW) algebra $D\otimes_{\C} \C[s,\partial_t]$ (or $D[u]\otimes_{\C} \C[s,\partial_t]$) with a non-commutative relation 
 $\partial_t s=s\partial_t-\partial_t$ 
and  commutative relations $\partial_t x_i=x_i \partial_t, \partial_i s=s \partial_i, s x_i=x_i s, \partial_t \partial_i=\partial_ i\partial_t$ ($1\le i \le n$). 
There exist algorithms and implementations to compute Gr\"obner bases of given ideals in the non-commutative rings $D$ and $D\langle s, \partial_t\rangle$ \cite{kw90,LM08}.

For $g_1,\ldots, g_r \in \C[u]$, $\V(g_1,\ldots,g_r) \subseteq \C^m$ denotes the affine variety of $g_1,\ldots, g_r$, i.e., $\V(g_1,\ldots,g_r)=\{\bar{u}\in \C^m | g_1(\bar{u})=\cdots =g_r(\bar{u})=0\}$. For $g_1,\ldots,g_r, g'_1,\ldots,g'_{r'} \in \C[t]$, we call an algebraic constructible set $\V(g_1,\ldots,g_r) \backslash \V(g'_1,\ldots,g'_{r'})\subseteq \C^m$  a stratum. Notations $\U_1, \U_2,$ $\ldots, \U_\ell$ are frequently used to represent strata.

For every $\bar{u} \in \C^m$, the canonical specialization homomorphism 
$\sigma_{\bar{u}} : D[u]\langle s, \partial_t\rangle \rightarrow D\langle s, \partial_t\rangle$ 
(or $D[u] \rightarrow D$) is defined as the map that substitutes $u$ by $\bar{u}$ in $p(u,x,\partial, s, $ $\partial_t) \in D[u]\langle s, \partial_t\rangle$. The image $\sigma_{\bar{u}}$ of a set $F$ is denoted by $\sigma_{\bar{u}}(F)=\{\sigma_{\bar{u}}(p) | p \in F\} \subset D\langle s, \partial_t\rangle$. 
A symbol $\Mono(x \cup \partial \cup \{s, \partial_t\})$ is the set of monomials of $x \cup \partial \cup \{s, \partial_s\}$.

The main tool to compute $b$-functions associated with $\mu$-constant deformations, is a comprehensive Gr\"obner system in the PBW algebra. We adopt the following as a definition of comprehensive Gr\"obner systems.

\begin{defn}[CGS]
Let $\succ$ be a monomial order on $\Mono(x \cup \partial \cup \{s,\partial_t\})$.  Let $F$ be a subset of $D[u]\langle s, \partial_t\rangle$, $\U_1, \U_2, \ldots, \U_\ell$ strata in $\C^m$ and $G_1,\ldots,G_\ell$ subsets in $D[u]\langle s, \partial_t\rangle$. If a finite set ${\mathcal G}=\{(\U_1,G_1),\ldots,(\U_\ell,G_\ell)\}$ of pairs satisfies properties such that  
\begin{enumerate}
\item $\U_i \neq \emptyset$ and $\U_i \cap \U_j= \emptyset$ for $1\le i\neq j \le \ell$, 

\item for all $\bar{u} \in \U_i$, $\sigma_{\bar{u}}(G_i)$ is a minimal Gr\"obner basis of $\Id(\sigma_{\bar{u}}(F))$ w.r.t. $\succ$ in $D\langle s, \partial_t\rangle$, and 
\item for all $\bar{u} \in \U_i$ and $p \in G_i$, $\sigma_{\bar{u}}(\hc(p))\neq 0$ where $\hc(p)$ is the head coefficient of $p$ in $\C[u]$,
\end{enumerate}
${\mathcal G}$ is called a comprehensive Gr\"obner system (CGS) on $\U_1 \cup \cdots \cup \U_\ell$ for $\Id(F)$ w.r.t. $\succ$. We simply say that ${\mathcal G}$ is a comprehensive Gr\"obner system for $\Id(F)$ if $\U_1 \cup \cdots \cup \U_\ell=\C^m$. 
\end{defn} 

In our previous papers \cite{NOT16,NOT18}, algorithms and implementations for computing comprehensive Gr\"obner systems in PBW algebras are introduced.

\subsection{Global $b$-functions}

Let $f$ be a non-constant polynomial in $\mathbb{C}[x]$. Then, the annihilating ideal of $f^s$ is 
\begin{center}
$\text{Ann}(f^s)= \{p \in  D[s] | pf^s = 0\}$ 
\end{center}
where $s$ is an indeterminate, and 
$$\displaystyle D[s]=\SetDef{\sum_{k \in \N, \beta \in \N^n}h_{k,\beta}(x)s^k \partial^\beta }{ h_{k,\beta}(x)\in \C[x]}.$$

Consider the following left ideal $I$ in the PBW algebra $D\langle s, \partial_t\rangle$.
\begin{center}
$I=\Id\left(f \cdot \partial_t+s, \partial_1+\partial_t\cdot \frac{\partial f}{\partial x_1}, \partial_2+\partial_t\cdot \frac{\partial f}{\partial x_2}, \ldots,\partial_n+\partial_t\cdot \frac{\partial f}{\partial x_n} \right).
$
\end{center}
Brian\c{c}on and Maisonobe show in \cite{BM} that 
$\text{Ann}(f^s)=I\cap D[s].$ 
Thus, a basis of the ideal $\text{Ann}(f^s)$ can be obtained by the Gr\"obner basis computation of $I$ w.r.t. an elimination order for $\partial_t$. 

The \textit{global $b$-function} or the global \textit{Bernstein-Sato polynomial} of $f$ is defined as the monic generator $b_f(s)$ of  $(\text{Ann}(f^s) + \Id(f))\cap  \mathbb{C}[s]$ where $\Id(f)$ is the ideal generated by $f$. It is known that the $b$-function of $f$ always has $s + 1$ as a factor and has a form $(s + 1)\tilde{b}_f (s)$, where $\tilde{b}_f (s) \in \mathbb{C}[s]$. The polynomial $\tilde{b}_f(s)$ is called the (global) \textit{reduced $b$-function} of $f$. The reduced $b$-function $\tilde{b}_f(s)$ can be obtained by computing a Gr\"obner basis of $\text{Ann}(f^s) + \Id(f,\frac{\partial f}{\partial x_1},\ldots,\frac{\partial f}{\partial x_n})$ w.r.t. a block order $\{x, \partial\} \gg s$.

Let $f$ be a parametric polynomial in $(\mathbb{C}[u])[x]$ where $u$ are regarded as parameters. 
As mentioned previously, a CGS of the $ \text{Ann}(f^s) $ is computable by the algorithm for computing CGS's in PBW algebras. Accordingly, global $b$-functions with parameters are computable by using CGS's of the ideals  $\text{Ann}(f^s) + \text{Id}(f)$ and $\text{Ann}(f^s) + \Id(f,\frac{\partial f}{\partial x_1},\ldots,\frac{\partial f}{\partial x_n})$. We refer the reader to \cite{NOT16} and \cite{NOT18} for details.

\noindent
\mbox{}\hrulefill\\
{\bf Algorithm 1}. (Global $b$-functions)\vspace{-2.0mm} \\
\mbox{}\hrulefill\\
{\bf Input:} $\U \subseteq \C^m$, $f \in ({\C}[u])[x]$: $f$ is a non-constant polynomial with parameter $u$, \\
{\bf Output:} ${\mathcal G}=\{(\U_1,\tilde{b}_1(s)), \ldots, (\U_\ell,\tilde{b}_\ell(s))\}$: for each $i \in \{1,\ldots, \ell\}$ and $\forall \bar{u} \in \U_i$, $\tilde{b}_i(s)$ is the reduced $b$-function of $\sigma_{\bar{u}}(f)$ and $\U=\bigcup_{i=1}^\ell \U_i$.\\
{\bf BEGIN} \\
${\mathcal G}\gets \emptyset$; \ $J\gets \{\frac{\partial f}{\partial x_1},\ldots,\frac{\partial f}{\partial x_n}\}$; \\ 
$I\gets \Id(f \cdot \partial_t+s, \partial_1+\partial_t\cdot \frac{\partial f}{\partial x_1}, \partial_2+\partial_t\cdot \frac{\partial f}{\partial x_2}, \ldots,\partial_n+\partial_t\cdot \frac{\partial f}{\partial x_n})$;\\
${\mathcal P} \gets$ Compute a CGS of $I$ w.r.t. an elimination order for $\partial_t$ on $\U$; \\
{\bf while} ${\mathcal P} \neq \emptyset$ {\bf do} \\
 \ \ \ Select $(\U', P)$ from ${\mathcal P}$; ${\mathcal P}\gets {\mathcal P}\backslash \{(\U', P)\}$; \\
 \ \ \ ${\mathcal B} \gets$ Compute a CGS of $\Id(J \cup (P\cap D[s]))$ on $\U'$ w.r.t. a block order $\{x, \partial\} \gg s$; \ \ \ $(\ast 1)$\\
 \ \ \ \ \ \ {\bf while} ${\mathcal B} \neq \emptyset$ {\bf do} \\
 \ \ \ \ \ \ \ \ \ Select $(\U'', B)$ from ${\mathcal B}$; ${\mathcal B}\gets {\mathcal B}\backslash \{(\U'', B)\}$; \\
 \ \ \ \ \ \ \ \ \  ${\mathcal G}\gets {\mathcal G}\cup \{(\U'',B\cap \C[s])\}$;\\
 \ \ \ \ \ \ {\bf end-while}\\
{\bf end-while} \\
{\bf return} ${\mathcal G}$;\\
{\bf END} \vspace{-3mm}\\
\noindent 
\mbox{}\hrulefill \vspace{-3mm} \\
\noindent 


Note that the rationality of the roots of $b$-functions has been shown by \cite{Kashi}. Hence, the $b$-functions can be factorized into linear factors over $\Q$.

\begin{exmp}\label{ex1}
Let us consider $U_{16}$ singularity $f=x^3+xz^2+y^5+u_1y^2z^{2}+u_2y^3z^{2}$ where $u_1, u_2$ are parameters. As $f$ is a semi-weighted homogeneous, the Milnor number of the singularity at the origin is 16. Set 
$
b_{\text{st}}(s)=(s+\frac{13}{15})(s+\frac{16}{15})(s+\frac{18}{15})(s+\frac{19}{15})(s+\frac{21}{15})(s+\frac{22}{15})(s+\frac{23}{15})(s+\frac{24}{15})(s+\frac{26}{15})(s+\frac{27}{15}).$ Then the reduced $b$-functions of $f$ is the following. 
\begin{enumerate}
\setlength{\leftskip}{-4mm}
\item[$\bullet$] If $(u_1,u_2)$ belongs to $\C^2\backslash \V(u_1(27u_1^4+256u_2))$, then $\tilde{b}_{f}(s)=b_{\text{st}}(s)(s+\frac{14}{15})(s+\frac{17}{15})$.
\item[$\bullet$] If $(u_1,u_2)$ belongs to $\V(27u_1^4+256u_2)\backslash \V(u_1,u_2)$, then $\tilde{b}_{f}(s)=b_{\text{st}}(s)(s+\frac{3}{2})(s+\frac{14}{15})(s+\frac{17}{15})$.
\item[$\bullet$] If $(u_1,u_2)$ belongs to $\V(u_1)\backslash \V(u_1,u_2)$, then $\tilde{b}_{f}(s)=b_{\text{st}}(s)(s+\frac{17}{15})(s+\frac{29}{15})$.
\item[$\bullet$] If $(u_1,u_2)$ belongs to $\V(u_1,u_2)$, then $\tilde{b}_{f}(s)=b_{\text{st}}(s)(s+\frac{29}{15})(s+\frac{32}{15})$.
\end{enumerate}

Note that if $(u_1,u_2)$ belongs to $\V(27u_1^4+256u_2)\backslash \V(u_1,u_2)$, then $f$ has three isolated singularities $(0,0,0)$ and $\displaystyle \left(\frac{3u_1^3}{64u_2^2},\frac{u_1}{4u_2},\pm\sqrt{\frac{-16}{27u_1^2u_2^2}}\right)$ in $\C^3$. In other cases, $f$ has one isolated singularity $(0,0,0)$.
\end{exmp}

\subsection{Local $b$-functions}\label{sb53}
Let 
\begin{center}
$\displaystyle {\mathcal D}[s]=\SetDef{\sum_{k \in \N, \beta \in \N^n}h_{k,\beta}(x)s^k\partial^\beta}{h_{k,\beta}(x)\in \C[x]_q}$
\end{center}
where 
$\C[x]_q=\{g_1(x)/g_2(x) | \ g_1(x),g_2(x)\in \C[x], g_2(q)\neq 0 \}$ 
the localization of $\C[x]$ at $q \in \C^n$. The local $b$-function of a non-constant polynomial $f \in \C[x]$ at $q$ is defined as the monic polynomial $b_{f,q}(s)$ of the minimal degree for $p \in {\mathcal D}[s]$ and $b_{f,q}(s) \in \C[s]$ satisfying $p\cdot f^{s+1}=a(x) b_{f,q}(s)\cdot f^s$ where $a(x) \in \C[x]_q$. The reduced $b$-function of $f$ at $q$, written as $\tilde{b}_{f,q}(s)$,  is $b_{f,q}(s)/(s+1)$.

In \cite{MN91}, Mebkhout and Narv\'aez-Macarro show the fact 
$$\tilde{b}_f(s)=\text{LCM}(\tilde{b}_{f,q}(s) | q \in \text{Sing}(f))$$ 
where $\text{Sing}(f)$ is the singular locus of $\V(f)$, i.e., $\text{Sing}(f)=\V(f,\frac{\partial f}{\partial x_1}, \ldots, \frac{\partial f}{\partial x_n})$.

We borrow from \cite{yano78} the following theorem.  

\begin{thm}\label{yano}
Let $f \in \C[x]$, $\gamma \in \Q$ and  set 
\begin{center}
$M_{(\gamma,f)}=D[s]/ (\text{Ann}(f^s)+\Id(f,\frac{\partial f}{\partial x_1},\ldots, \frac{\partial f}{\partial x_n})+\Id(s-\gamma)).$
\end{center}
Then, if $\tilde{b}_f(\gamma)\neq 0$, then $M_{(\gamma,f)}=\{0\}$, and 
if $\tilde{b}_f(\gamma)= 0$, then $M_{\gamma}$ is a holonomic $D$-module and $\text{supp}(M_{(\gamma,f)})\subseteq \text{Sing}(f)$ where $\text{supp}(M_{(\gamma,f)})$ is the support of $M_{(\gamma,f)}$.
\end{thm}

\begin{lem}\label{lem1}
Using the same notation as in Theorem~\ref{yano}, let $\gamma$ be a root of $\tilde{b}_{f,q}(s)=0$. Then, $q \in \text{supp}(M_{(\gamma,f)})$.
\end{lem}

Algorithms and implementations for computing $M_{(\gamma,f)}$ and $\text{supp}(M_{(\gamma,f)})$, have been already introduced in \cite{LM08} and in \cite{NOT16,NOT18,Oaku1}. \\

After here, we consider local $b$-functions of semi-weighted homogeneous polynomials at $O \in \C^n$. For a semi-weighted homogeneous polynomial, the following property is known.  

\begin{thm}[C.2.1.6. \cite{B89}]\label{single}
For a semi-weighted homogeneous polynomial $f\in \C[x]$, the local $b$-function $ \tilde{b}_{f,0}(s) $ is square-free, i.e.,  $\tilde{b}_{f,0}(s)=0$ has no multiple roots.
\end{thm} 

Note that, there is a possibility that the global $b$-function of a semi-weighted homogeneous polynomial has multiple roots, whereas as Theorem~\ref{single} says, the local $b$-function of a semi-weighted homogeneous polynomial has no multiple roots. \\

We turn to parametric cases. Let $f$ be a semi-weighted homogeneous polynomial in $(\C[u])[x]$. Then, by utilizing Lemma~\ref{lem1} and Theorem~\ref{single}, we are able to construct an algorithm for computing local $b$-functions of $f$ at $O \in \C^n$. 

\noindent
\mbox{}\hrulefill\\
{\bf Algorithm 2}. (Local $b$-functions at $O$)\vspace{-2.0mm} \\
\mbox{}\hrulefill\\
{\bf Input:} $\U \subseteq \C^m$, $f=f_0+g \in (\C[u])[x]$: $\forall \bar{u}\in \U$, $\sigma_{\bar{u}}(f_0)$ is a weighted homogeneous polynomial with an isolated singularity at $O$ and $\sigma_{\bar{u}}(f)$ is a semi-weighted homogeneous polynomial, \\
{\bf Output:} ${\mathcal G}_0=\{(\U_1,\tilde{b}_1(s)), \ldots, (\U_\ell,\tilde{b}_\ell(s))\}$: for each $i \in \{1,\ldots, \ell\}$ and $\forall \bar{u} \in \U_i$, $\tilde{b}_i(s)$ is the local reduced $b$-function of $\sigma_{\bar{u}}(f)$ at $O$ (i.e., $\tilde{b}_i(s)=\tilde{b}_{\sigma_{\bar{u}}(f), 0}(s)$) and  $\U=\bigcup_{i=1}^\ell \U_i$. \\
{\bf BEGIN} \\
${\mathcal G}_0\gets \emptyset$; \ ${\mathcal G} \gets$ Execute Algorithm~1 on $\U$;\\
{\bf while} ${\mathcal G} \neq \emptyset$ {\bf do} \\
 \ \ \ Select $(\U', b(s))$ from ${\mathcal G}$; ${\mathcal G}\gets {\mathcal G}\backslash \{(\U', b(s))\}$; \\
 \ \ \ $E \gets $ Compute all roots of $b(s)=0$;\\
 \ \ \ $\tilde{b}\gets 1;$\\
 \ \ \ \ \ \ {\bf while} $E \neq \emptyset$ {\bf do} \\
 \ \ \ \ \ \ \ \ \ Select $\gamma$ from $E$; $E\gets E\backslash \{\gamma\}$; \\
 \ \ \ \ \ \ \ \ \ {\bf if} $O \in \text{supp}(M_{(\gamma,f)})$ {\bf then} \\
 \ \ \ \ \ \ \ \ \ \ \ \ $\tilde{b} \gets \tilde{b}\cdot (s-\gamma)$; \\
 \ \ \ \ \ \ \ \ \ {\bf end-if} \\
 \ \ \ \ \ \ {\bf end-while}\\
 \ \ \ ${\mathcal G}_0\gets {\mathcal G}_0\cup \{(\U',\tilde{b})\}$;\\
{\bf end-while}\\
{\bf return} ${\mathcal G}_0$;\\
{\bf END} \vspace{-3mm} \\
\noindent 
\mbox{}\hrulefill \vspace{-3mm} \\
\noindent

\begin{exmp}
Let us consider Example~\ref{ex1}, again. The Milnor number $\mu$ of the singularity $x^3+xz^2+y^5=0$, at the origin $O$,  is 16, and the $\mu$-constant deformation is given by $f=x^3+xz^2+y^5+u_1y^2z^2+u_2y^3z^2$ where $u_1,u_2$ are parameters. 

 According to Algorithm~2, we compute the local $b$-function $\tilde{b}_{f,0}(s)$. We consider the  case where $ (u_1, u_2) \in \V(27u_1^4+256u_2)\backslash \V(u_1,u_2) $ and examine the holonomic $ D $-module $ M_{(-\frac{3}{2}, f)} $ associated with the factor $ s+\frac{3}{2} $ of $\tilde{b}_{f}(s)$. 
Direct computation shows $$ \text{supp}(M_{(-\frac{3}{2},f)}) = \V(64u_2^2x-3u_1^3,4u_2y-u_1,16u_2^3z^2-u_1^2).$$ 
Since $O \notin \text{supp}(M_{(-\frac{3}{2},f)})$, $ s+\frac{3}{2} $ is not a factor of $\tilde{b}_{f,0}(s)$. The supports associated with other roots, contain the origin $O$. Therefore, we obtain Table~\ref{u162} as $\tilde{b}_{f,0}(s)$ where $b_{\text{st}}(s)$ is from Example~\ref{ex1}.

\begin{table}[ht]
\begin{center}
\caption{List of local $b$-functions}~\label{u162}
{\renewcommand\arraystretch{1.2}
\begin{tabular}{|l|l|} 
 \hline
stratum & (Local) \ $\tilde{b}_{f,0}(s)$ \\ \hline
$\C^2\backslash \V(u_1)$ &
$b_{\text{st}}(s)(s+\frac{14}{15})(s+\frac{17}{15})$ \\\hline
$\V(u_1)\backslash \V(u_2)$ & 
$b_{\text{st}}(s)(s+\frac{17}{15})(s+\frac{29}{15})$ \\\hline
$\V(u_1,u_2)$ & 
$b_{\text{st}}(s)(s+\frac{29}{15})(s+\frac{32}{15})$ \\\hline
\end{tabular}
}
\end{center} 
\end{table}
\end{exmp} 

As the example above shows that Algorithm~2 gives a computation method of $b$-functions associated with $\mu$-constant deformations. We implemented Algorithm~2 in the computer algebra system {\sf Risa/Asir} \cite{NT92}, and we tried to compute all problems of Table~\ref{modality2} for {\bf three months} by using 3 computers: PC1 [OS: Linux, CPU: Xeon E3-1225, 3.2 GHz, Memory: 126 GB], PC2 [OS: Linux, CPU: Xeon E3-1230, 3.3 GHz, Memory: 504 GB] and PC3 [OS: Windows~10, CPU: Core i7-5930k, 3.5 GHz, Memory: 64 GB].

All bases of the annihilating ideals of $f^s$ in Table~\ref{modality2} were successfully obtained. 
It turned out by this computer experiment that the cost of computation of the part $(\ast 1)$ of Algorithm~1 ({\it i.e.}, computing a CGS of $\text{Ann}(f^s)+\Id(f, \frac{\partial f}{\partial x_1}, \ldots, \frac{\partial f}{\partial x_n})$) w.r.t. an elimination  order) is quite high. Among 20 cases, our implementation could not return 12 $b$-functions  within {\bf three months} and output only 8 $b$-functions : $b$-functions of $E_{18}$, $E_{20}$, $W_{18}$, $Z_{17}$, $Q_{16}$, $Q_{17}$, $S_{17}$, $U_{16}$. We see  that the direct use of Algorithm~2 is not adequate for computing $b$-functions associated with $ \mu$-constant deformations. 
In order to overcome  difficulties, we improve the method presented in this section by specializing Algorithm~2 to handle semi-weighted homogeneous cases.

\section{Computing local cohomology solutions to a holonomic $D$-module}

Here we introduce an algorithm for computing local cohomology solutions of the holonomic $D$-module $ M_{(\gamma, f)}$. The algorithm will be utilized as a key tool in the new computation method of $b$-functions.

All local cohomology classes, in this paper, are algebraic local cohomology classes that belong to the set defined by 
$$H_{[O]}^n(\C[x])=\lim_{k\rightarrow \infty}\text{Ext}^n_{\C[x]}(\C[x]/\langle x_1,x_2,\ldots, x_n\rangle^k, \C[x])$$ 
where $\langle x_1,x_2,\ldots, x_n\rangle$ is the maximal ideal generated by $x_1,\ldots, x_n$. We adopt notations used in \cite{NT17} to represent algebraic local cohomology classes, namely, we represent an algebraic local cohomology class $\sum c_{\lambda}
\left[\begin{array}{c}
 1  \\
 x^{\lambda+1}
 \end{array} 
 \right]$ as a polynomial $\sum c_{\lambda}\xi^\lambda$ where $\xi$ is the abbreviation of $n$ variables $\xi_1,\ldots,\xi_n$, $c_{\lambda} \in \C$ and $\lambda=(\lambda_1,\ldots, \lambda_n) \in \N^n$. The multiplication is defined as
$$x^{\alpha}\ast \xi^\lambda
=\left\{
\begin{array}{ll}
 \xi^{\lambda-\alpha}, \ \ \ \ \ \ \ \lambda_i\ge \alpha_i, i=1,\ldots,n,\\
\ \\
 0,  \ \ \ \ \ \ \ \  \ \ \ \ \ \ \ \ \text{otherwise,} 
\end{array}
\right.
$$
where $\alpha=(\alpha_1,\alpha_2,\ldots, \alpha_n) \in \N^n$ and $\lambda-\alpha=(\lambda_1-\alpha_1,\ldots, \lambda_n-\alpha_n) \in \N^n$. The partial derivative by $\frac{\partial}{\partial x_i}$ is defined as
\begin{center}
$
\frac{\partial}{\partial x_i}\ast(
 \xi_1^{\lambda_1}\xi_2^{\lambda_2}\cdots \xi_i^{\lambda_i}\cdots \xi_n^{\lambda_n}) =-(\lambda_i+1)
 \xi_1^{\lambda_1}\xi_2^{\lambda_2}\cdots \xi_i^{\lambda_i+1}\cdots \xi_n^{\lambda_n}.
$ 
\end{center}

Let fix a monomial order $\succ$. For a given algebraic local cohomology class of the form
$$\psi=c_{\lambda}\xi^{\lambda}+\sum_{\xi^{\lambda} \succ \xi^{\lambda'}}c_{\lambda'}\xi^{\lambda'}, \ \ c_{\lambda}\neq 0$$
we call $\xi^\lambda$ the {\it head monomial},  $c_\lambda$ the {\it head coefficient} and $\xi^{\lambda'}$ the {\it lower monomials}. We write the head monomial as $\hht(\psi)$.\\

Let $f$ be a holomorphic function defined on an open neighborhood $X$ of the origin $O$ of the $n$-dimensional complex space $\C^n$, with an isolated singularity at the origin. Let $\gamma$ be a root of the local reduced $b$-function $\tilde{b}_{f,0}(s)$ at $O$. Let $G'$ be a minimal Gr\"obner basis of $\text{Ann}(f^s)+\Id(f, \frac{\partial f}{\partial x_1}, \ldots, $ $\frac{\partial f}{\partial x_n})+\Id(s-\gamma)$ w.r.t. a monomial order satisfying $\{x,\partial\} \gg s$ in $D[s]$. Set $G_{(\gamma,f)}=G'\backslash \{s-\gamma\}$, then $M_{(\gamma,f)}=D/\Id(G_{(\gamma,f)})$. 
We define a set $H_{M_{(\gamma,f)}}$ to be the set of algebraic local cohomology classes in $H^n_{[O]}(\C[x])$ that are annihilated by $G_{(\gamma,f)}$:
$$
H_{M_{(\gamma,f)}}=\SetDef{\psi \in H^n_{[O]}(\C[x])}{h \ast \psi=0, \forall h \in G_{(\gamma,f)}}.
$$

Since $ H_{M_{(\gamma, f)}} $ is the algebraic local cohomology solution space of the holonomic $D$-module $ M_{(\gamma, f)}, $ we have the following. 

\begin{thm}\label{prob}
The set $H_{M_{(\gamma,f)}}$ is a finite dimensional vector space. 
\end{thm}

Here we introduce an algorithm for computing a basis of the vector space $H_{M_{(\gamma,f)}}$. 

\begin{lem}\label{hodai}
Using the same notation as in above, let $P_0=G_{(\gamma,f)} \cap \C[x]$, $F_0=P_0\cup \{f, \frac{\partial f}{\partial x_1}, \ldots, $ $\frac{\partial f}{\partial x_n}\} \subset \C[x]$. Set 
$$H_{F_0}=\{\psi \in H^n_{[O]}(\C[x]) | h \ast \psi=0, \forall h \in F_0\}.$$
Then, $H_{M_{(\gamma,f)}} \subseteq H_{F_0}$.
\end{lem}

Since $P_0 \subset G_{(\gamma,f)}$, Lemma~\ref{hodai} holds. Note that $F_0 \subset \C[x]$, thus, a basis of the vector space $H_{F_0}$ can be obtained by the algorithm \cite{NT17,TNN09}. An algorithm for computing a basis of $H_{M_{(\gamma,f)}}$ is the following. 

\noindent
\mbox{}\hrulefill\\
{\bf Algorithm 3}. ( A basis of $H_{M_{(\gamma,f)}}$ )\vspace{-2.0mm} \\
\mbox{}\hrulefill\\
{\bf Input:} $f \in \C[x]$: a polynomial with an isolated singularity at $O$. $\gamma \in \Q$: a root of $\tilde{b}_{f,0}(s)$.  Fix a monomial order $\succ$ on $\C[\xi]$.\\
{\bf Output:} $\Psi$: a basis of the vector space $H_{M(\gamma,f)}$.\\
{\bf BEGIN} \\
$G_{(\gamma,f)} \gets $ Compute $G_{(\gamma,f)}$; \ $\{f_1,\ldots,f_r\}\gets G_{(\gamma,f)}\cap \C[x]$; \\
 $F_0\gets \{f_1,\ldots,f_r\}\cup \{f,\frac{\partial f}{\partial x_1},\ldots,\frac{\partial f}{\partial x_n}\}$; \ $P_1\gets G_{(\gamma,f)}\backslash \{f_1,\ldots,f_r\}$; \\
$G_0\gets$ Compute a basis of the vector space $H_{F_0}$;  \ \ \ \ \verb|/*| $G_0$ \verb|is echelon form. */|\\
$\Psi \gets \emptyset$; \ $L\gets \emptyset$; \\ 
{\bf while} $G_0\neq \emptyset$ {\bf do} \\
 \ \ \ Select $\psi$ whose head monomial is the smallest in $\hht(G_0)$; $G_0\gets G_0 \backslash \{\psi\}$; \\ 
 \ \ \ $\displaystyle \varphi \gets \psi+\sum_{\varsigma_i \in L} c_{i}\varsigma_i $;  \ \ \ \ \ \ \ \verb|/*The symbol| $c_i$ \verb|is an indeterminate.*/|\\
 \ \ \ $E\gets $ Make a system of linear equations with $c_i$ from $\{p \ast \psi =0 | p \in P_1\}$; \\
 \ \ \ {\bf if} $E$ has a solution {\bf then}\\
 \ \ \ \ $\varphi' \gets$ Substitute the solution into $c_i$ of $\varphi$;\\
 \ \ \ \ $\Psi \gets \Psi\cup \{\varphi' \}$; \\ 
 \ \ \ {\bf else} \\
 \ \ \ \  $L \gets L \cup \{\psi \}$; \ \ \ \ \ \ \  \ \ \ \verb|/* candidate of lower cohomology classes.*/|\\
 \ \ \ {\bf end-if} \\
{\bf end-while} \\
{\bf return} $\Psi$;\\ 
{\bf END}\vspace{-3mm} \\
\mbox{}\hrulefill \vspace{-3mm} 
 \ \\

\begin{thm}
Algorithm~3 returns a basis of the vector space $H_{M(\gamma,f)}$ and terminates.
\end{thm} 
\begin{proof}
As $H_{F_0}$ is the finite dimensional vector space, the set $G_0$ is finite. Thus, this algorithm terminates.  
Since each element in the output $ \Psi $ satisfies linear partial differential  equation $ G_{(\gamma, f)}, $ we have $ \Psi \subset H_{M_{(\gamma, f)}}.$
 Since $  H_{M_{(\gamma, f)}} \subset H_{F_0}, $ we have  $\text{Span}(\Psi) = H_{M_{(\gamma, f)}}. $ Furthermore, each element in $ \Psi $ has a form $\displaystyle  \psi+\sum_{\varsigma_i \in L} c_{i}\varsigma_i$, they are linearly independent.  Therefore, the algorithm returns a basis of the vector space $ H_{M_{(\gamma, f)}}. $
\end{proof}

We illustrate Algorithm~3 with the following example

\begin{exmp}
Let us consider $f=x^3+yz^2+y^7+xy^5+xz^2 \in \C[x,y,z]$ that defines an isolated singularity at the origin. By computing $\tilde{b}_{f,0}$, we have rational numbers $\gamma=-\frac{19}{21}$ and $-\frac{4}{3}$ as roots of $\tilde{b}_{f,0}(\gamma)=0$. Let us execute Algorithm~3 to get bases of the vector spaces $H_{M_{(-\frac{19}{21}, f)}}$ and $H_{M_{(-\frac{4}{3},f)}}$.  
Let $\xi, \eta, \zeta$ denote symbols correspond to the variables $x, y, z$, and $\partial_x=\frac{\partial}{\partial x}$, $\partial_y=\frac{\partial}{\partial y}$, $\partial_z=\frac{\partial}{\partial z}$.
 The monomial order $\succ$  is the degree lexicographic with $\xi \succ \eta \succ \zeta$.

\begin{enumerate}
\setlength{\leftskip}{-3mm}
\item[$\bullet$] In case $\gamma=-\frac{19}{21}$, then $G_{(-\frac{19}{21},f)}=\{x,y,z\}$. Thus, it is obvious that $H_{M_{(-\frac{19}{21},f)}}=\text{Span}\left(1\right)$. 
 
\item[$\bullet$] In case $\gamma=-\frac{4}{3}$, then 
$G_{(-\frac{4}{3},f)}=\{z^2,yx,xz,xy,x^2,12x-y^3,24x\partial_y+x+2y^2,81xz\partial_z-1367x+32y^3-1152y^2\partial_y-48y^2-768yz\partial_z-5376y,41472x\partial_x
+61503x+1064y^2+13824y\partial_y+576y+41472z\partial_z+138240,35378x+124416y^2\partial_x+184509y^2-497664y\partial_y^2-41472y\partial_y+75744y-1492992z\partial_y\partial_z-62208z\partial_z-3981312\partial_y-165888\}$.

Let $P_0=G_{(-\frac{4}{3},f)}\cap \C[x,y,z]$ and 
\begin{center}
$F_0=P_0\cup \{f, \frac{\partial f}{\partial x},\frac{\partial f}{\partial y},\frac{\partial f}{\partial z}\}=\{z^2,zy,zx,yx,x^2,12x-y^3,f, \frac{\partial f}{\partial x},\frac{\partial f}{\partial y},\frac{\partial f}{\partial z}\}$  
\end{center}
A basis $G_0$ of the vector space of $H_{F_0}$ is 
$G_0=\{1, \zeta,\eta, \eta^2,\eta^3+\frac{1}{12}\xi \}$. Set \\
$P_1=G_{(-\frac{4}{3},f)}\backslash P_0=\{24x\partial_y+x+2y^2,81xz\partial_z-1367x+32y^3-1152y^2\partial_y-48y^2-768yz\partial_z-5376y,41472x\partial_x+61503x+1064y^2+13824y\partial_y+576y+41472z\partial_z+138240,35378x+124416y^2\partial_x+184509y^2-497664y\partial_y^2-41472y\partial_y+75744y-1492992z\partial_y\partial_z-62208z\partial_z-3981312\partial_y-165888\}$, $\Psi=\emptyset$ and $L=\emptyset$.
\item[(1)] Take $1$ whose head monomial is the smallest in $\hht(G_0)$ w.r.t. $\succ$. Renew $G_0$ as $G_0\backslash \{1\}$ and take $41472x\partial_x+61503x+1064y^2+13824y\partial_y+576y+41472z\partial_z+138240$ from $P_1$. Then, it is clear that $(41472x\partial_x+61503x+1064y^2+13824y\partial_y+576y+41472z\partial_z+138240)\ast (1)\neq 0$. Thus, renew $L$ as $L\cup \{1\}$.

\item[(2)] Take $\zeta$ whose head monomial is the smallest in $\hht(G_0)$  w.r.t. $\succ$. Renew $G_0$ as $G_0\backslash \{\zeta\}$. Set $\varphi=\zeta+c$ where $c$ is an indeterminate. Then $\{p\ast  \varphi|p\in P_1\}=\{0, 0, 41472c, -62208c+1492992c\eta \}.$
Thus, when $c=0$, then $\varphi \in H_{M_{(-\frac{4}{3},f)}}$. Renew $\Psi$ as $\Psi \cup \{\zeta\}$.

\item[(3)] Take $\eta$ whose head monomial is the smallest in $\hht(G_0)$. Renew $G_0$ as $G_0\backslash \{\eta\}$. Set $\varphi=\eta+c$ where $c$ is an indeterminate. Then, 
$\{p\ast \varphi |p\in P_1\}=\{
0,-2304, (41472c+576)+27648\eta,(75744-62208c)+(1492992c-20736)\eta\}.$ Obviously, the second element $-2304$ is not zero. Thus, renew $L$ as $\{\eta,1\}$.

\item[(4)] Take $\eta^2$ whose head monomial is the smallest in $\hht(G_0)$. Renew $G_0$ as $G_0\backslash \{\eta^2\}$. Set $\varphi=\eta^2+c_1\eta+c_2$ where $c_1,c_2$ are indeterminates. Then, $(24x\partial_y+x+2y^2)\ast \varphi=2 \neq 0$ where $24x\partial_y+x+2y^2 \in P_1$. Hence, $\varphi \notin H_{M_{(-\frac{4}{3},f)}}$ and renew $L$ as $\{\eta^2, \eta,1\}$.
 
\item[(5)] Take $\eta^3+\frac{1}{12}\xi$ from $G_0$. Renew $G_0$ as $G_0\backslash \{\eta^3+\frac{1}{12}\xi \}$ where $c_1,c_2, c_3$ are indeterminates. Set $\varphi =\eta^3+\frac{1}{12}\xi+c_1\eta^2+c_2\eta+c_3$. Then, \\
$\{p\ast \varphi |p\in P_1\}=\{2c_1+\frac{1}{12}, 48c_1+2304c_2+\frac{266}{3}+(1152c_1+48)\eta,1064c_1+576c_2+41472c_3+\frac{20501}{4}+(576c_1+27648c_2+1064)\eta+(13824c_1+576)\eta^2,184$ $509c_1+75744c_2-62208c_3+\frac{17689}{6}+(75744c_1-20736c_2+1492992c_3+184509)\eta-(124416c_1+5184)\eta^2+(20736c_1+1990656c_2+75744)\eta^2+(1492992c_1+62208)\eta^3\}$ \\
Solve the following system of linear equations that are from $\{p\ast \varphi |p\in P_1\}.$\\
$2c_1+\frac{1}{12}=0, 48c_1+2304c_2+\frac{266}{3}=0,1152c_1+48=0,1064c_1+576c_2+41472c_3+\frac{20501}{4}=0,576c_1+27648c_2+1064=0,13824c_1+576=0,184509c_1+75744c_2-62208c_3+\frac{17689}{6}=0,75744c_1-20736c_2+1492992c_3+184509=0,124416c_1+5184=0,20736c_1+1990656c_2+75744=0,1492992c_1+62208=0.$\\
Then, $c_1=-\frac{1}{24}, c_2=-\frac{65}{1728}, c_3=0$. Renew $\Psi$ as $\{\eta, \eta^3+\frac{1}{12}\xi-\frac{1}{24}\eta^2-\frac{65}{1728}\eta\}$.

\item[(6)] Since $G_0 = \emptyset$, Algorithm~3 stops. Therefore, $\Psi$ is a basis of the vector space $H_{M_{(-\frac{4}{3},f)}}$.
\end{enumerate}
\end{exmp}

By using a framework presented in \cite{NT17}, we have extended Algorithm~3 to handle parametric cases. The resulting algorithm that compute the  parametric local cohomology solution space of parametric holonomic $ D$-module $ M_{(\gamma, f)} $ is implemented in the computer algebra system {\sf Risa/Asir}.

\section{New algorithm}\label{s5}
Here we introduce a new algorithm for computing reduced $b$-functions $\tilde{b}_{f,0}$ of a semi-weighted homogeneous polynomial. As we described in subsection~\ref{sb53}, the computational complexity of computing a Gr\"obner basis of $\text{Ann}(f^s)+\Id(f,\frac{\partial f}{\partial x_1}, \ldots, \frac{\partial f}{\partial x_n})$ is quite high. In order to overcome the difficulty, we adopt the idea introduced by  Levandovskyy and Martin-Morales \cite{LM12} and address to the computation of $b$-functions associated with $ \mu$-constant deformations.

Let $f=f_0+g \in \C[x]$ where $ f_0 $ is the weighted homogeneous part,  $ d \in {\mathbb N}$, ${\bf w}=(w_1,\ldots, w_n) \in \N^n$  and  $g$ is a linear combination of upper monomials.

\subsection{Properties of semi-weighted homogeneous singularities}
First we review some properties of semi-weighted homogeneous singularities that  are needed for constructing a new algorithm. 

\begin{defn}
Let $ f_0$ be a weighted homogeneous polynomial of type $(d; {\bf w})$ with an isolated singularity at the origin $ O. $ The Poincar\'e polynomial of $ f_0 $ is the univariate polynomial defined to be 
$$P_{(d;{\bf w})}(t)=\frac{t^{d-w_1}-1}{t^{w_1}-1}\cdot \frac{t^{d-w_2}-1}{t^{w_2}-1}\cdots \frac{t^{d-w_n}-1}{t^{w_n}-1}.$$
\end{defn} 

It is well-known that all roots of $\tilde{b}_{f_0,0}$ can be computed by the Poincar\'e polynomial.

\begin{thm}\label{root}
Let $f_0$ be a weighted homogeneous polynomial of type $(d;{\bf w})$ with an isolated singularity at the origin $O$, and $ w_0=\sum_{i=1}^nw_i$. Let $ P_{(d;{\bf w})}(t)=\sum_{i=1}^rc_it^{\alpha_i}$ ($c_i\neq 0$) be the Poincar\'e polynomial of type $(d; {\bf w})$. Then, the set of roots of $\tilde{b}_{f_0,0}(s)=0$ is equal to $\SetDef{-\frac{\alpha_i+w_0}{d}}{ 1\le i \le r}$.
\end{thm}

The following properties are from \cite{kashi75,yano78}.

\begin{thm}\label{dim}
Let $\{\gamma_1,\ldots,\gamma_r\}$ be the set of roots of $\tilde{b}_{f,0}(s)=0$ and let $\mu$ be the Milnor number of the singularity $f=0$ at the origin $O$. Then,  $\displaystyle  \mu=\sum_{i=1}^r\dim_{\C}(H_{M_{(\gamma_i,f)}}).$

\end{thm}

\subsection{New algorithm}

The new algorithm mainly consists of the  following three steps,
\begin{enumerate}
\setlength{\leftskip}{5mm}
\item[Step~1:] to compute candidates $\gamma$s of the roots of $\tilde{b}_{f,0}(s)=0$,
\item[Step~2:] to check whether $\tilde{b}_{f,0}(\gamma)=0$ or $\tilde{b}_{f,0}(\gamma) \neq 0$,
\item[Step~3:] to check whether $\Gamma=\{\gamma_1,\ldots,\gamma_r\}$ is equal to $\{\gamma | \tilde{b}_{f,0}(\gamma)=0\}$ or not, where $\gamma_1,\ldots,\gamma_r$ are roots of $\tilde{b}_{f,0}(s)=0$.
\end{enumerate}

The following lemma quoted from \cite{Kashi}, tells us how to compute the candidates. 
\begin{lem}
Let $E_{0}=\{\gamma \in \mathbb{Q} | \tilde{b}_{f_0,0}(\gamma)=0\}$ and $E=\{\gamma \in \mathbb{Q} | \tilde{b}_{f,0}(\gamma)=0\}$.
Then, $E$ is a subset of $E'=\{\gamma+k | \gamma \in E_{0}, k\in \mathbb{Z}, -n<\gamma +k <0\}$ 
 where $\mathbb{Z}$ is the set of integers. 
\end{lem} 

Since $ E_0 $ is determined by Theorem~\ref{root}, it is easy to obtain $ E^{\prime}. $ Empirically, it is sufficient to check $k=0,1,2$. Hence, in Step~1, we use $$E'=\{\gamma+k | \gamma \in E_{0}, k\in \{0,1,2\}, -n<\gamma +k <0\}$$ as s set of candidates of the roots.
  
Next, in Step~2, we have to check whether $\gamma \in E'$ is a root of $\tilde{b}_{f,0}(s)=0$ or not. We borrow the idea of Levandovskyy and Martin-Morales \cite{LM12}.

\begin{lem}\label{check}
Let $H$ be a basis of $\text{Ann}(f^{s})$ in $D[s]$ and $\gamma \in \Q$. Let $G$ be a minimal Gr\"obner basis of $\Id(H\cup \{f, \frac{\partial f}{\partial x_1},\frac{\partial f}{\partial x_2},\ldots,$ $\frac{\partial f}{\partial x_n} \} \cup \{s-\gamma\})$ w.r.t. a block order with $x\cup \partial \gg s$. Then, if $s-\gamma \in G$, $s-\gamma$ is a factor of the global $b$-function of $f$.
\end{lem}
\noindent 
{\bf Remark:} The computational speed of computing a minimal Gr\"obner basis of $\Id(H\cup \{f, \frac{\partial f}{\partial x_1},\ldots,$ $\frac{\partial f}{\partial x_n} \} \cup \{s-\gamma\})$ is much faster than that of $\Id(H\cup \{f, \frac{\partial f}{\partial x_1},\ldots,\frac{\partial f}{\partial x_n} \})$, because the degree of $s-\gamma$ is 1. Moreover, for a candidate $\gamma \in E'$, it is sufficient to check the single root because of Theorem~\ref{single}. However, as we are considering ``local'' $b$-functions at the origin $O$, we need to check the support of $\gamma$ (i.e., $\text{supp}(M_{(\gamma,f)})$) if $s-\gamma \in G$.\\

As we know how to compute a CGS in $D[s]$, we can naturally  extend the idea to parametric cases. We have implemented the parametric version of Lemma~\ref{check} in {\sf Risa/Asir}.\\

In our implementation, the command \verb|para_ann1| (or \verb|para_ann|) returns a CGS of $\text{Ann}(f^s)$,  and the command \verb|root_check| (or \verb|root_check11|) returns a CGS of $\Id(\text{Ann}(f^s)\cup \{f, \frac{\partial f}{\partial x_1},\frac{\partial f}{\partial x_2},\ldots,\frac{\partial f}{\partial x_n} \} \cup \{s-\gamma\})$ where $\gamma \in \Q$. In {\em Example}~4, \verb|ANN| is a CGS of $\text{Ann}(f^s)$. The form \verb|[[S1],[S2]]| means a stratum $\V($\verb|S1|$)\backslash \V($\verb|S2|$)$. 

\begin{exmp} Let us consider $S_{16}$ singularity. The $\mu$-constant deformation  is given by $f=x^2z+yz^2+xy^4+u_1y^6+u_2z^3$ where $u_1,u_2$ are parameters. Let us check whether $17s+19$ is a factor of $\tilde{b}_f(s)$ or not. 
{\small 
\begin{verbatim}
[2727] F=x^2*z+y*z^2+x*y^4+u1*y^6+u2*z^3$
[2728] ANN=para_ann1(F,[u1,u2],[x,y,z])$
[2729] roots_check(ANN,F,17*s+19,[u1,u2],[x,y,z]);
[[0],[486*u2*u1^13-1143*u2^2*u1^9-639*u2^3*u1^5-68*u2^4*u1]]
[17*s+19,z,x,y^2,(4030*u1^4+4913*dy*u1-9826*u2)*y+11560*u2^2*z+9826*u1]
 
[[27*u1^4+4*u2],[u1,u2]]
[17*s+19,z,x,-y^2,(140711*u1^3+9826*dy)*y+19652]
 
[[u1,u2],[1]]
[1]
 
[[6*u1^4-17*u2],[u1,u2]]
[17*s+19,z,x,y^2,(-562*u1^3-4913*dy)*y-9826]
 
[[3*u1^4+u2],[u1,u2]]
[17*s+19,z,x,y^2,(33508*u1^3+4913*dy)*y+9826]
 
[[u1],[u1,u2]]
[17*s+19,z,y,x]
 
[[u2],[u1,u2]]
[17*s+19,z,x,y^2,(4030*u1^3+4913*dy)*y+9826]
\end{verbatim}
}

The monomial order $\succ$ used in the computation above is a block order $\{\partial_x, \partial_y, \partial_z\} \gg \{x,y,z\} \gg s$ which is specified on  $\text{Mono}(\{\partial_x, \partial_y, \partial_z\})$ as the total degree lexicographic monomial order with $\partial_x \succ \partial_y \succ \partial_z$, and on $\Mono(\{x,y,z\})$ as the total degree lexicographic monomial order with $z \succ  y \succ  x$.

Let $I$ denote the ideal generated by $\text{Ann}(f^s)$ and $\{f, \frac{\partial f}{\partial x},\frac{\partial f}{\partial y},\frac{\partial f}{\partial z}, 17s+19\}$ in the ring $(D[u_1,u_2])[s]$, i.e., $=\Id(\text{Ann}(f^s) \cup \{f, \frac{\partial f}{\partial x},\frac{\partial f}{\partial y},\frac{\partial f}{\partial z}, 17s+19\})$.

The meanings of the output above are the following.
\begin{enumerate}
\item If $(u_1,u_2)$ belongs to $\U_1=\left(\C^2 \backslash \V(486u_2u_1^{13}-1143u_2^2u_1^9-639u_2^3u_1^5-68u_2^4u_1)\right)=\left(\C^2 \backslash \V((27u_1^4+\right.$ $\left.4u_2)(6u_1^4-17u_2)(3u_1^4+u_2)u_1u_2)\right)$, then 
$$G_1=\{17s+19,z,x,y^2,(4030u_1^4+4913\partial_y u_1-9826u_2)y+11560u_2^2z+9826u_1\}$$
is a minimal Gr\"obner basis of $I$ w.r.t. $\succ$. As $17s+19 \in G_1$ and $G_1\cap \mathbb{C}[x,y,z]=\{z,x,y^2\}$, $17s+19$ is a factor of $\tilde{b}_f(s)$ and $\text{supp}(M_{(-\frac{19}{17},f)})=\{O\}$, namely, $17s+19$ is a factor of $\tilde{b}_{f,0}(s)$. 

\item If $(u_1,u_2)$ belongs to $\U_2=\V(27u_1^4+4u_2) \backslash \V(u_1,u_2)$, then 
$G_2=\{17s+19,z,x,-y^2,(140711$ $u_1^3+9826\partial_y)y+19652\}$ 
is a minimal Gr\"obner basis of $I$ w.r.t. $\succ$. As $17s+19 \in G_2$ and $G_2\cap \C[x,y,z]=\{z,x,-y^2\}$, $17s+19$ is a factor of $\tilde{b}_{f,0}(s)$. 

\item If $(u_1,u_2)$ belongs to $\U_3=\V(u_1,u_2)$, then $G_3=\{1\}$
is a minimal Gr\"obner basis of $I$ w.r.t. $\succ$. As $17s+19 \notin G_3$,  $17s+19$ is not a factor of $\tilde{b}_f(s)$. 

\item If $(u_1,u_2)$ belongs to $\U_4=\V(6u_1^4-17u_2) \backslash \V(u_1,u_2)$, then 
$G_4=\{17s+19,z,x,y^2,(-562u_1^3-4913\partial_y)y-9826\}$ 
is a minimal Gr\"obner basis of $I$ w.r.t. $\succ$. As $17s+19 \in G_4$ and $G_4\cap \C[x,y,z]=\{z,x,y^2\}$, $17s+19$ is a factor of $\tilde{b}_{f,0}(s)$. 

\item If $(u_1,u_2)$ belongs to $\U_5=\V(3u_1^4+u_2) \backslash \V(u_1,u_2)$, then 
$G_5=\{17s+19,z,x,y^2,(33508u_1^3+4913\partial_y)y+9826\}$ 
is a minimal Gr\"obner basis of $I$ w.r.t. $\succ$. As $17s+19 \in G_5$ and $G_5\cap \C[x,y,z]=\{z,x,y^2\}$, $17s+19$ is a factor of $\tilde{b}_{f,0}(s)$. 

\item If $(u_1,u_2)$ belongs to $\U_6=\V(u_1) \backslash \V(u_1,u_2)$, then 
$G_6=\{17s+19,z,y,x\}$ 
is a minimal Gr\"obner basis of $I$ w.r.t. $\succ$. As $17s+19 \in G_6$ and $G_6\cap \C[x,y,z]=\{z,y,x\}$, $17s+19$ is a factor of $\tilde{b}_{f,0}(s)$. 

\item If $(u_1,u_2)$ belongs to $\U_7=\V(u_2) \backslash \V(u_1,u_2)$, then 
$G_7=\{17s+19,z,x,y^2,(4030u_1^3+4913\partial_y)y+9826\}$ 
is a minimal Gr\"obner basis of $I$ w.r.t. $\succ$. As $17s+19 \in G_7$ and $G_7\cap \C[x,y,z]=\{z,x,y^2\}$, $17s+19$ is a factor of $\tilde{b}_{f,0}(s)$. 

\end{enumerate}

Note that  $17s+19 \in G_i \ (i=1,2,4,5,6,7)$ and $17s+19 \notin G_3$.

Since 
$$\U_1\cup \U_2 \cup \U_4 \cup \U_5 \cup \U_6 \cup \U_7=\C^2 \backslash \U_3, $$ we have \\
i) if $(u_1,u_2)\neq (0,0)$, then $17s+19$ is a factor of $\tilde{b}_f(s)$ and $\text{supp}(M_{(-\frac{19}{17},f)})=\{O\}$, namely, $17s+19$ is a factor of (local) $\tilde{b}_{f,0}(s)$, \\
ii) if $u_1=u_2=0$, then $17s+19$ is not a factor of (global) $\tilde{b}_f(s)$.
\end{exmp}

In Step 3, we apply Theorem~\ref{dim} and we use Algorithm~3 for computing $ {\rm dim}_{{\mathbb C}}(H_{M_{(\gamma, f)}}). $

\noindent
\mbox{}\hrulefill\hrulefill\\
{\bf Algorithm 4}. (Local $b$-function at $O$)\vspace{-2.0mm} \\
\mbox{}\hrulefill\hrulefill\\
{\bf Input:} $f=f_0+g \in \C[x]$: a semi-weighted homogeneous polynomial with an isolated singularity at $O$. \\
$H$: a basis of $\text{Ann}(f^s)$.  \ 
$\mu(f_0) \in \N$: the Milnor number of $f_0$ at the origin $O$.\\
$\succ$ : a monomial order satisfying $\{x,\partial\} \gg s$. \\
{\bf Output:} $b$: the reduced $b$-function $\tilde{b}_{f,0}(s)$.\\
{\bf BEGIN} \\
$b\gets 1$; \ $\mu\gets 0$; \ $k\gets 0$; \ $J \gets \{f,\frac{\partial f}{\partial x_1},\ldots ,\frac{\partial f}{\partial x_n}\}$; \\
$E_0 \gets$ Compute all roots of $\tilde{b}_{f_0,0}(s)$ by the Poincar\'e polynomial; \\
{\bf while} $\mu \neq \mu(f_0)$ {\bf do}\\
 \ \ \ $E' \gets\{\gamma+k | \gamma \in E_{0}, -n<\gamma +k <0\}$; \\
\ \ \ {\bf while} $E'\neq \emptyset$ {\bf do} \\
 \ \ \ \ \ \ Select $\gamma$ from $E'$; $E'\gets E' \backslash \{\gamma\}$; \\
 \ \ \ \ \ \ $G \gets $ Compute a minimal Gr\"obner basis of $\Id(H\cup J \cup \{s-\gamma\})$ w.r.t. $\succ$;\\
 \ \ \ \ \ \ {\bf if} $\left( s-\gamma \in G\right)$ and $\left(O \in \text{supp}(M_{(\gamma,f)})\right)$ {\bf then} \\
 \ \ \ \ \ \ \ \ \ $b\gets b\cdot (s-\gamma)$; \\
 \ \ \ \ \ \ \ \ \ $\mu \gets \mu+\dim_{\C}(H_{M_{(\gamma,f)}})$;\\
 \ \ \ \ \ \ {\bf end-if}\\
\ \ \ {\bf end-while}\\
\ \ \ $k\gets k+1$;\\
{\bf end-while}\\
{\bf return} $b$;\\
{\bf END} \vspace{-3mm} \\
\noindent 
\mbox{}\hrulefill\hrulefill \vspace{-3mm} \\
\noindent 

As we described in section~3 and 4, a CGS of $\text{Ann}(f^s)$ and a basis of the vector space $H_{M_{(\gamma,f)}}$ with parameters, are computable. Thus, Algorithm~4 can be generalized to parametric cases, too. We have computed all $b$-functions of Table~\ref{modality2} by utilizing the generalized algorithm. We have obtained all $b$-functions of Table~\ref{modality2}, successfully.  The $b$-functions of Table~\ref{modality2} are given in Section~6.\\
 \ \\
{\bf Remark:} We tried to compute local $b$-functions associated with $\mu$-constant deformations of inner modality 3 singularities by using a computer PC1 (see section 3). We could obtain none of CGSs of $\text{Ann}(f^s)$ within three months and thus we could not use the method described in section 5. 
We expect however that local $b$-functions associated with $\mu$-constant deformations can be computed by using the proposed method provided that parametric bases of $\text{Ann}(f^s)$ is given. \\

In this paper, we introduce a new algorithm for computing parametric $\tilde{b}_{f,0}$  of a semi-weighted homogeneous polynomial by improving  Algorithm~2. The resulting algorithm has better performance than Algorithm~2. 
The resulting algorithm and Algorithm~2 use the same CGS algorithm  described in \cite{NOT16,NOT18}. The difference lies in the way of its using. The versatility of CGS algorithm allows the specialization and the improvement.

\section{List of $b$-functions associated with $\mu$-constant deformations }

Here all $b$-functions $\tilde{b}_{f,0}(s)$ of $\mu$-constant deformation of inner modality 2 singularities, are presented. 

Currently, our {\sf Risa/Asir} implementation is in the following webpage
\begin{center}
\url{https://www-math.ias.tokushima-u.ac.jp/~nabesima/bfunction/bfunc2.html}. 
\end{center}
Two computation times ``CGS of $\text{Ann}(f^s)$'' and ``parametric version of Algorithm 4'' are also presented in each $\mu$-constant deformation. The time is given in CPU seconds. The computer [OS: Windows 10, CPU: intel core i9-7900X, 3.30 GHz, Memory: 128 GB] was used. 

\begin{enumerate}
\setlength{\leftskip}{0.3cm}
\item[$\bullet$  $E_{18}$:] $f=x^3+y^{10}+u_1xy^7+u_2xy^8$ 
\begin{flushleft}
{\renewcommand\arraystretch{1.2}
\begin{tabular}{|c|c|}
\hline 
stratum & $\tilde{b}_{f,0}(s)$ \\ \hline
$\C^2\backslash \V(u_1)$ &
 $b_{\text{st}}(s)(s+\frac{14}{30})(s+\frac{17}{30})$ \\ \hline
$\V(u_1)\backslash \V(u_1,u_2)$ & 
 $b_{\text{st}}(s)(s+\frac{17}{30})(s+\frac{44}{30})$  \\ \hline
$\V(u_1,u_2)$ & $b_{\text{st}}(s)(s+\frac{44}{30})(s+\frac{47}{30})$  \\ \hline
\end{tabular}
}
\end{flushleft}
\begin{flushleft}
\setlength{\leftskip}{-0.9cm}
$b_{\text{st}}(s)=(s+\frac{13}{30})(s+\frac{16}{30})(s+\frac{19}{30})(s+\frac{22}{30})(s+\frac{23}{30})(s+\frac{25}{30})(s+\frac{26}{30})$ \\
\hspace{0cm}$\times (s+\frac{28}{30})(s+\frac{29}{30})(s+\frac{31}{30})(s+\frac{32}{30})(s+\frac{34}{30})(s+\frac{35}{30})(s+\frac{37}{30})(s+\frac{38}{30})(s+\frac{41}{30}).$
\end{flushleft}
\noindent 
\begin{tabular}{cccc}
CGS of $\text{Ann}(f^s)$: 3141 &  \ \ \ & parametric version of Algorithm~4: 5.156 &  \ \ \  \\
\end{tabular}

\item[$\bullet$  $E_{19}$:]$f=x^3+xy^{7}+u_1y^{11}+u_2y^{12}$ 
\begin{flushleft}
{\renewcommand\arraystretch{1.2} 
\begin{tabular}{|c|c|}
\hline 
$\C^2\backslash \V(u_1)$ & $b_{\text{st}}(s)(s+\frac{10}{21})(s+\frac{12}{21})$ \\ \hline
$\V(u_1)\backslash \V(u_1,u_2)$ & $b_{\text{st}}(s)(s+\frac{12}{21})(s+\frac{31}{21})$ \\ \hline
$\V(u_1,u_2)$ &$b_{\text{st}}(s)(s+\frac{31}{21})(s+\frac{33}{21})$ \\ \hline
\end{tabular} 
} 
\end{flushleft}
\begin{flushleft}
\setlength{\leftskip}{-0.9cm}
$b_{\text{st}}(s)=(s+\frac{9}{21})(s+\frac{11}{21})(s+\frac{13}{21})(s+\frac{15}{21})(s+\frac{16}{21})(s+\frac{17}{21})(s+\frac{18}{26})(s+\frac{19}{21})$ \\
\hspace{0.0cm}$\times (s+\frac{20}{21})(s+\frac{21}{21})(s+\frac{22}{21})(s+\frac{23}{21})(s+\frac{24}{21})(s+\frac{25}{21})(s+\frac{26}{21})(s+\frac{27}{21})(s+\frac{29}{21}).$
\end{flushleft}
\noindent 
\begin{tabular}{cccc}
CGS of $\text{Ann}(f^s)$: 10410 &  \ \ \ & parametric version of Algorithm~4: 45.5 &  \ \ \  \\
\end{tabular}

\item[$\bullet$  $E_{20}$:] $f=x^3+y^{11}+u_1xy^8+u_2xy^9$ 
\begin{flushleft}
{\renewcommand\arraystretch{1.2}
\begin{tabular}{|c|c|}
\hline 
$\C^2\backslash \V(u_1)$ & $b_{\text{st}}(s)(s+\frac{16}{33})(s+\frac{19}{33})$  \\ \hline
$\V(u_1)\backslash \V(u_1,u_2)$ &  $b_{\text{st}}(s)(s+\frac{19}{33})(s+\frac{49}{33})$ \\ \hline
$\V(u_1,u_2)$ & $b_{\text{st}}(s)(s+\frac{49}{33})(s+\frac{52}{33})$ \\\hline
\end{tabular} 
}
\end{flushleft}
\begin{flushleft}
\setlength{\leftskip}{-0.9cm}
$b_{\text{st}}(s)=(s+\frac{14}{33})(s+\frac{17}{33})(s+\frac{20}{33})(s+\frac{23}{33})(s+\frac{25}{33})(s+\frac{26}{33})(s+\frac{28}{33})(s+\frac{29}{33})(s+\frac{31}{33})$ \\
\hspace{0.1cm}$\times (s+\frac{32}{33})(s+\frac{34}{33})(s+\frac{35}{33})(s+\frac{37}{33})(s+\frac{38}{33})(s+\frac{40}{33})(s+\frac{41}{33})(s+\frac{43}{33})(s+\frac{46}{33}).$
\end{flushleft}
\noindent 
\begin{tabular}{cccc}
CGS of $\text{Ann}(f^s)$: 2667 &  \ \ \ & parametric version of Algorithm~4: 164 &  \ \ \  \\
\end{tabular}

\item[$\bullet$ $W_{17}$:] $f=x^4+xy^{5}+u_1y^{7}+u_2y^{8}$
\begin{flushleft}
{\renewcommand\arraystretch{1.2}
\begin{tabular}{|c|c|}
\hline
$\C^2\backslash \V(u_1)$ & $b_{\text{st}}(s)(s+\frac{9}{20})(s+\frac{12}{20})$ \\\hline
$\V(u_1)\backslash \V(u_1,u_2)$ &  $b_{\text{st}}(s)(s+\frac{12}{20})(s+\frac{29}{20})$ \\ \hline
$\V(u_1,u_2)$ &  $b_{\text{st}}(s)(s+\frac{29}{20})(s+\frac{32}{20})$ \\ \hline
\end{tabular} 
}
\end{flushleft}
\begin{flushleft}
\setlength{\leftskip}{-0.9cm}
$b_{\text{st}}(s)=(s+\frac{8}{20})(s+\frac{11}{20})(s+\frac{13}{20})(s+\frac{14}{20})(s+\frac{16}{20})(s+\frac{17}{20})(s+\frac{18}{20})(s+\frac{19}{20})$ \\
\hspace{0.0cm}$\times (s+\frac{20}{20})(s+\frac{21}{20})(s+\frac{22}{20})(s+\frac{23}{20})(s+\frac{24}{20})(s+\frac{26}{20})(s+\frac{27}{20}).$
\end{flushleft}
\noindent 
\begin{tabular}{cccc}
CGS of $\text{Ann}(f^s)$: 354.1 &  \ \ \ & parametric version of Algorithm~4: 7.547 &  \ \ \  \\
\end{tabular}

\item[$\bullet$ $W_{18}$:]$f=x^4+y^7+u_1x^2y^4+u_2x^2y^5$
\begin{flushleft}
{\renewcommand\arraystretch{1.2}
\begin{tabular}{|c|c|}
\hline 
$\C^2\backslash \V(u_1)$ & $b_{\text{st}}(s)(s+\frac{13}{28})(s+\frac{17}{28})$  \\\hline
$\V(u_1)\backslash \V(u_1,u_2)$ &  $b_{\text{st}}(s)(s+\frac{17}{28})(s+\frac{41}{28})$  \\ \hline
$\V(u_1,u_2)$ & $b_{\text{st}}(s)(s+\frac{41}{28})(s+\frac{45}{28})$ \\ \hline\end{tabular} 
}
\end{flushleft}
\begin{flushleft}
\setlength{\leftskip}{-0.9cm}
$b_{\text{st}}(s)=(s+\frac{11}{28})(s+\frac{15}{28})(s+\frac{18}{28})(s+\frac{19}{28})(s+\frac{22}{28})(s+\frac{23}{28})(s+\frac{25}{28})(s+\frac{26}{28})$ \\
\hspace{0.0cm}$\times (s+\frac{27}{28})(s+\frac{29}{28})(s+\frac{30}{28})(s+\frac{31}{28})(s+\frac{33}{28})(s+\frac{34}{28})(s+\frac{37}{28})(s+\frac{38}{28}).$\end{flushleft}
\noindent 
\begin{tabular}{cccc}
CGS of $\text{Ann}(f^s)$: 149.8 &  \ \ \ & parametric version of Algorithm~4: 2.438 &  \ \ \  \\
\end{tabular}

\item[$\bullet$  $Z_{17}$:]  $f=x^3y+y^8+u_1xy^6+u_2xy^7$
\begin{flushleft}
{\renewcommand\arraystretch{1.2}
\begin{tabular}{|c|c|}
\hline 
$\C^2\backslash \V(u_1)$ & $b_{\text{st}}(s)(s+\frac{11}{24})(s+\frac{14}{24})$ \\\hline
$\V(u_1)\backslash \V(u_1,u_2)$ &  $b_{\text{st}}(s)(s+\frac{14}{24})(s+\frac{35}{24})$ \\\hline
$\V(u_1,u_2)$ & $b_{\text{st}}(s)(s+\frac{35}{24})(s+\frac{38}{24})$ \\\hline
 \end{tabular} 
}
\end{flushleft}
\begin{flushleft}
\setlength{\leftskip}{-0.9cm}
$b_{\text{st}}(s)=(s+\frac{10}{24})(s+\frac{13}{24})(s+\frac{16}{24})(s+\frac{17}{24})(s+\frac{19}{24})(s+\frac{20}{24})(s+\frac{22}{24})(s+\frac{23}{24})$ \\
\hspace{0cm}$\times (s+\frac{24}{24})(s+\frac{25}{24})(s+\frac{26}{24})(s+\frac{28}{24})(s+\frac{29}{24})(s+\frac{31}{24})(s+\frac{32}{24}).$
\end{flushleft}
\noindent 
\begin{tabular}{cccc}
CGS of $\text{Ann}(f^s)$: 903.8 &  \ \ \ & parametric version of Algorithm~4: 15.25 &  \ \ \  \\
\end{tabular}

\item[$\bullet$  $Z_{18}$ :] $f=x^3y+xy^6+u_1y^9+u_2y^{10}$
\begin{flushleft}
{\renewcommand\arraystretch{1.2}
\begin{tabular}{|c|c|}
\hline
$\C^2\backslash \V(u_1)$ & $b_{\text{st}}(s)(s+\frac{8}{17})(s+\frac{10}{17})$ \\ \hline
$\V(u_1)\backslash \V(u_1,u_2)$ &  $b_{\text{st}}(s)(s+\frac{10}{17})(s+\frac{25}{17})$ \\ \hline
$\V(u_1,u_2)$ & $b_{\text{st}}(s)(s+\frac{25}{17})(s+\frac{27}{17})$ \\ \hline\end{tabular} 
}
\end{flushleft}
\begin{flushleft}
\setlength{\leftskip}{-0.9cm}
$b_{\text{st}}(s)=(s+\frac{7}{17})(s+\frac{9}{17})(s+\frac{11}{17})(s+\frac{12}{17})(s+\frac{13}{17})(s+\frac{14}{17})(s+\frac{15}{17})(s+\frac{16}{17})$ \\
\hspace{0cm}$\times (s+\frac{17}{17})(s+\frac{18}{17})(s+\frac{19}{17})(s+\frac{20}{17})(s+\frac{21}{17})(s+\frac{22}{17})(s+\frac{23}{17}).$
\end{flushleft}
\noindent 
\begin{tabular}{cccc}
CGS of $\text{Ann}(f^s)$: 546.1 &  \ \ \ & parametric version of Algorithm~4: 5.344 &  \ \ \  \\
\end{tabular}

\item[$\bullet$ $Z_{19}$: ]  $f=x^3y+y^9+u_1xy^7+u_2xy^8$
\begin{flushleft}
{\renewcommand\arraystretch{1.2}
\begin{tabular}{|c|c|}
\hline
$\C^2\backslash \V(u_1)$ & $b_{\text{st}}(s)(s+\frac{13}{27})(s+\frac{16}{27})$ \\ \hline
$\V(u_1)\backslash \V(u_1,u_2)$ &  $b_{\text{st}}(s)(s+\frac{16}{27})(s+\frac{40}{27})$ \\\hline
$\V(u_1,u_2)$ & $b_{\text{st}}(s)(s+\frac{40}{27})(s+\frac{43}{27})$ \\\hline
\end{tabular} 
}
\end{flushleft}
\begin{flushleft}
\setlength{\leftskip}{-0.9cm}
$b_{\text{st}}(s)=(s+\frac{11}{27})(s+\frac{14}{27})(s+\frac{17}{27})(s+\frac{19}{27})(s+\frac{20}{27})(s+\frac{22}{27})(s+\frac{23}{27})(s+\frac{25}{27})$ \\
\hspace{0cm}$\times (s+\frac{26}{27})(s+\frac{27}{27})(s+\frac{28}{27})(s+\frac{29}{27})(s+\frac{31}{27})(s+\frac{32}{27})(s+\frac{34}{27})(s+\frac{35}{27})(s+\frac{37}{27}).$
\end{flushleft}
\noindent 
\begin{tabular}{cccc}
CGS of $\text{Ann}(f^s)$: 984.2 &  \ \ \ & parametric version of Algorithm~4: 5.141 &  \ \ \  \\
\end{tabular}

\item[$\bullet$  $Q_{16}$:] $f=x^3+yz^2+y^7+u_1xy^5+u_2xz^2$
\begin{flushleft}
{\renewcommand\arraystretch{1.2}
\begin{tabular}{|c|c|}
\hline
$\C^2\backslash \V(u_1)$ & $b_{\text{st}}(s)(s+\frac{20}{21})(s+\frac{23}{21})$ \\\hline
$\V(u_1)\backslash \V(u_1,u_2)$ &  $b_{\text{st}}(s)(s+\frac{23}{21})(s+\frac{41}{21})$ \\\hline
$\V(u_1,u_2)$ & $b_{\text{st}}(s)(s+\frac{41}{21})(s+\frac{44}{21})$  \\\hline
 \end{tabular} 
}
\end{flushleft}
\begin{flushleft}
\setlength{\leftskip}{-0.9cm}
$b_{\text{st}}(s)=(s+\frac{19}{21})(s+\frac{22}{21})(s+\frac{25}{21})(s+\frac{26}{21})(s+\frac{28}{21})(s+\frac{29}{21})(s+\frac{31}{21})(s+\frac{32}{21})$ \\
\hspace{0cm}$\times (s+\frac{34}{21})(s+\frac{35}{21})(s+\frac{37}{21})(s+\frac{38}{21}).$
\end{flushleft}
\noindent 
\begin{tabular}{cccc}
CGS of $\text{Ann}(f^s)$: 50730 &  \ \ \ & parametric version of Algorithm~4: 3.953 &  \ \ \  \\
\end{tabular}

\item[$\bullet$  $Q_{17}$:] $f=x^3+yz^2+xy^5+u_1y^8+u_2y^9$
\begin{flushleft}
{\renewcommand\arraystretch{1.2}
\begin{tabular}{|c|c|}
\hline 
$\C^2\backslash \V(u_1)$ & $b_{\text{st}}(s)(s+\frac{29}{30})(s+\frac{33}{30})$ \\\hline
$\V(u_1)\backslash \V(u_1,u_2)$ &  $b_{\text{st}}(s)(s+\frac{33}{30})(s+\frac{59}{30})$  \\\hline
$\V(u_1,u_2)$ & $b_{\text{st}}(s)(s+\frac{59}{30})(s+\frac{63}{30})$ \\\hline
 \end{tabular} 
}
\end{flushleft}
\begin{flushleft}
\setlength{\leftskip}{-0.9cm}
$b_{\text{st}}(s)=(s+\frac{27}{30})(s+\frac{31}{30})(s+\frac{35}{30})(s+\frac{37}{30})(s+\frac{39}{30})(s+\frac{40}{30})(s+\frac{41}{30})(s+\frac{43}{30})$ \\
\hspace{0cm}$\times (s+\frac{45}{30})(s+\frac{47}{30})(s+\frac{49}{30})(s+\frac{50}{30})(s+\frac{51}{30})(s+\frac{53}{30})(s+\frac{55}{30}).$
\end{flushleft}
\noindent 
\begin{tabular}{cccc}
CGS of $\text{Ann}(f^s)$: 1073 &  \ \ \ & parametric version of Algorithm~4: 2.813 &  \ \ \  \\
\end{tabular}

\item[$\bullet$ $Q_{18}$ :] $f=x^3+yz^{2}+y^8+u_1xy^{6}+u_2xz^{2}$
\begin{flushleft}
{\renewcommand\arraystretch{1.2}
\begin{tabular}{|c|c|}
\hline
$\C^2\backslash \V(u_1)$ & $b_{\text{st}}(s)(s+\frac{47}{48})(s+\frac{53}{48})$  \\\hline
$\V(u_1)\backslash \V(u_1,u_2)$ & $b_{\text{st}}(s)(s+\frac{53}{48})(s+\frac{95}{48})$ \\\hline
$\V(u_1,u_2)$ & $b_{\text{st}}(s)(s+\frac{95}{48})(s+\frac{101}{48})$ \\\hline\end{tabular} 
}
\end{flushleft}
\begin{flushleft}
\setlength{\leftskip}{-0.9cm}
$b_{\text{st}}(s)=(s+\frac{43}{48})(s+\frac{49}{48})(s+\frac{55}{48})(s+\frac{59}{48})(s+\frac{61}{48})(s+\frac{64}{48})(s+\frac{65}{48})(s+\frac{67}{48})$ \\
\hspace{0cm}$\times (s+\frac{71}{48})(s+\frac{73}{48})(s+\frac{77}{48})(s+\frac{79}{48})(s+\frac{80}{48})(s+\frac{83}{48})(s+\frac{85}{48})(s+\frac{89}{48}).$
\end{flushleft}
\noindent 
\begin{tabular}{cccc}
CGS of $\text{Ann}(f^s)$: 1516 &  \ \ \ & parametric version of Algorithm~4: 6.844 &  \ \ \  \\
\end{tabular}

\item[$\bullet$  $S_{16}$ :] $f=x^2z+yz^2+xy^4+u_1y^6+u_2z^3$
\begin{flushleft}
{\renewcommand\arraystretch{1.2}
\begin{tabular}{|c|c|}
\hline
 $\C^2\backslash \V(u_1)$ & $b_{\text{st}}(s)(s+\frac{16}{17})(s+\frac{19}{17})$ \\\hline
$\V(u_1)\backslash \V(u_1,u_2)$ &  $b_{\text{st}}(s)(s+\frac{19}{17})(s+\frac{33}{17})$ \\\hline
$\V(u_1,u_2)$ &  $b_{\text{st}}(s)(s+\frac{33}{17})(s+\frac{36}{17})$ \\\hline \end{tabular} 
}
\end{flushleft}
\begin{flushleft}
\setlength{\leftskip}{-0.9cm}
$b_{\text{st}}(s)=(s+\frac{15}{17})(s+\frac{18}{17})(s+\frac{20}{17})(s+\frac{21}{17})(s+\frac{22}{17})(s+\frac{23}{17})(s+\frac{24}{17})$ \\
\hspace{0cm}$\times (s+\frac{25}{17})(s+\frac{26}{17})(s+\frac{27}{17})(s+\frac{28}{17})(s+\frac{29}{17})(s+\frac{30}{17})(s+\frac{31}{17}).$
\end{flushleft}
\noindent 
\begin{tabular}{cccc}
CGS of $\text{Ann}(f^s)$: 7107 &  \ \ \ & parametric version of Algorithm~4: 18.98 &  \ \ \  \\
\end{tabular}

\item[$\bullet$  $S_{17}$ :] $f=x^2z+yz^2+y^6+u_1y^4z+u_2z^3$
\begin{flushleft}
{\renewcommand\arraystretch{1.2}
\begin{tabular}{|c|c|}
\hline 
 $\C^2\backslash \V(u_1)$ & $b_{\text{st}}(s)(s+\frac{23}{24})(s+\frac{27}{24})$ \\\hline
$\V(u_1)\backslash \V(u_1,u_2)$ &  $b_{\text{st}}(s)(s+\frac{27}{24})(s+\frac{47}{24})$ \\\hline
$\V(u_1,u_2)$ &  $b_{\text{st}}(s)(s+\frac{47}{24})(s+\frac{51}{24})$ \\\hline
 \end{tabular} 
}
\end{flushleft}
\begin{flushleft}
\setlength{\leftskip}{-0.9cm}
$b_{\text{st}}(s)=(s+\frac{21}{24})(s+\frac{25}{24})(s+\frac{28}{24})(s+\frac{29}{24})(s+\frac{31}{24})(s+\frac{32}{24})(s+\frac{33}{24})(s+\frac{35}{24})$ \\
\hspace{0cm}$\times (s+\frac{36}{24})(s+\frac{37}{24})(s+\frac{39}{24})(s+\frac{40}{24})(s+\frac{41}{24})(s+\frac{43}{24})(s+\frac{44}{24}).$
\end{flushleft}
\noindent 
\begin{tabular}{cccc}
CGS of $\text{Ann}(f^s)$: 26220 &  \ \ \ & parametric version of Algorithm~4: 3.063 &  \ \ \  \\
\end{tabular}

\item[$\bullet$  $U_{16}$:] $x^3+xz^2+y^5+u_1y^2z^2+u_2y^3z^2$
\begin{flushleft}
{\renewcommand\arraystretch{1.2}
\begin{tabular}{|c|c|}
\hline
$\C^2\backslash \V(u_1)$ & $b_{\text{st}}(s)(s+\frac{14}{15})(s+\frac{17}{15})$ \\\hline
$\V(u_1)\backslash \V(u_1,u_2)$ &  $b_{\text{st}}(s)(s+\frac{17}{15})(s+\frac{29}{15})$  \\\hline
$\V(u_1,u_2)$ & $b_{\text{st}}(s)(s+\frac{29}{15})(s+\frac{32}{15})$ \\\hline
 \end{tabular} 
}
\end{flushleft}
\begin{flushleft}
\setlength{\leftskip}{-0.9cm}
$b_{\text{st}}(s)=(s+\frac{13}{15})(s+\frac{16}{15})(s+\frac{18}{15})(s+\frac{19}{15})(s+\frac{21}{15})(s+\frac{22}{15})(s+\frac{23}{15})(s+\frac{24}{15})(s+\frac{26}{15})(s+\frac{27}{15})$ 
\end{flushleft}
\noindent 
\begin{tabular}{cccc}
CGS of $\text{Ann}(f^s)$: 3141 &  \ \ \ & parametric version of Algorithm~4: 5.156 &  \ \ \  \\
\end{tabular}

\item[$\bullet$ $J_{16}$:] $f=x^3+y^{9}+u_1x^2y^{3}+u_2y^{10}$ \ \ ($4u_1^3+27 \neq 0$) 
\begin{flushleft}
{\renewcommand\arraystretch{1.2}
\begin{tabular}{|c|c||c|}
\hline
stratum & $\tilde{b}_{f,0}(s)$ &  note \\ \hline
$\C^2\backslash \V((4u_1^3+27)u_1u_2)$ & $b_{\text{st}}(s)$ & $\dim_{\C}(H_{M_{(-5/9,f)}})=2$   \\\hline
$\V(u_1u_2)\backslash \V(4u_1^3+27)$ &  $b_{\text{st}}(s)(s+\frac{14}{9})$ & $\dim_{\C}(H_{M_{(-5/9,f)}})=1,$ \\
 & &   $\dim_{\C}(H_{M_{(-14/9,f)}})=1$ \\ \hline
\end{tabular} 
}
\end{flushleft}
\begin{flushleft}
\setlength{\leftskip}{-0.9cm}
$b_{\text{st}}(s)=(s+\frac{4}{9})(s+\frac{5}{9})(s+\frac{6}{9})(s+\frac{7}{9})(s+\frac{8}{9})(s+\frac{9}{9})(s+\frac{10}{9})(s+\frac{11}{9})(s+\frac{12}{9})(s+\frac{13}{9}).$
\end{flushleft}
\noindent 
\begin{tabular}{cccc}
CGS of $\text{Ann}(f^s)$: 55.83 &  \ \ \ & parametric version of Algorithm~4: 304.6 &  \ \ \  \\
\end{tabular}

\item[$\bullet$ $W_{15}$ :] $f=x^4+y^{6}+u_1x^2y^{3}+u_2y^{7}$ \ \ ($u_1^2-4 \neq 0$) 
\begin{flushleft}
{\renewcommand\arraystretch{1.2}
\begin{tabular}{|c|c||c|}
\hline 
$\C^2\backslash \V((u_1^2-4)u_1u_2)$ & $b_{\text{st}}(s)$ & $\dim_{\C}(H_{M_{(-7/12,f)}})=2$  \\\hline
$\V(u_1u_2)\backslash \V(u_1^2-4)$ &  $b_{\text{st}}(s)(s+\frac{19}{12})$ & $\dim_{\C}(H_{M_{(-7/12,f)}})=1,$  \\
 & & $\dim_{\C}(H_{M_{(-19/12,f)}})=1$ \\ \hline
\end{tabular} 
}
\end{flushleft}
\begin{flushleft}
\setlength{\leftskip}{-0.9cm}
$b_{\text{st}}(s)=(s+\frac{5}{12})(s+\frac{7}{12})(s+\frac{8}{12})(s+\frac{9}{12})(s+\frac{10}{12})(s+\frac{11}{12})(s+\frac{12}{12})(s+\frac{13}{12})$ \\
\hspace{0cm}$\times (s+\frac{14}{12})(s+\frac{15}{12})(s+\frac{16}{12})(s+\frac{17}{12}).$
\end{flushleft}
\noindent 
\begin{tabular}{cccc}
CGS of $\text{Ann}(f^s)$: 16.88 &  \ \ \ & parametric version of Algorithm~4: 158 &  \ \ \  \\
\end{tabular}

\item[$\bullet$ $Z_{15}$ :] $f=x^3y+y^{7}+u_1x^2y^{3}+u_2y^{8}$ \ \ ($4u_1^3+27 \neq 0$)
\begin{flushleft}
{\renewcommand\arraystretch{1.2}
\begin{tabular}{|c|c||c|}
\hline 
$\C^2\backslash \V((4u_1^3+27)u_1u_2)$ & $b_{\text{st}}(s)$ & $\dim_{\C}(H_{M_{(-4/7,f)}})=2$ \\\hline
$\V(u_1u_2)\backslash \V(4u_1^3+27)$ &  $b_{\text{st}}(s)(s+\frac{11}{7})$ & $\dim_{\C}(H_{M_{(-4/7,f)}})=1$ \\
 & &  $\dim_{\C}(H_{M_{(-11/7,f)}})=1$ \\\hline
\end{tabular} 
}
\end{flushleft}
\begin{flushleft}
\setlength{\leftskip}{-0.9cm}
$b_{\text{st}}(s)=(s+\frac{3}{7})(s+\frac{4}{7})(s+\frac{5}{7})(s+\frac{6}{7})(s+\frac{7}{7})(s+\frac{8}{7})(s+\frac{9}{7})(s+\frac{10}{7}).$
\end{flushleft}
\noindent 
\begin{tabular}{cccc}
CGS of $\text{Ann}(f^s)$: 36.86 &  \ \ \ & parametric version of Algorithm~4: 767.2 &  \ \ \  \\
\end{tabular}

\item[$\bullet$ $Q_{14}$:] $f=x^3+yz^{2}+u_1x^2y^{2}+xy^4+u_2y^{7}$ \ \ ($u_1^2-4 \neq 0$)
\begin{flushleft}
{\renewcommand\arraystretch{1.2}
\begin{tabular}{|c|c||c|}
\hline
$\C^2\backslash \V((u_1^2-4)(4u_1^2-15)u_2)$ & $b_{\text{st}}(s)$ \ \ \ & $\dim_{\C}(H_{M_{(-13/12,f)}})=2$   \\\hline
$\V(u_2(4u_1^2-15))\backslash \V(u_1^2-4)$ &  $b_{\text{st}}(s)(s+\frac{25}{12})$ & $\dim_{\C}(H_{M_{(-13/12,f)}})=1,$  \\
 & & $\dim_{\C}(H_{M_{(-25/12,f)}})=1$ \\ \hline
\end{tabular} 
}
\end{flushleft}
\begin{flushleft}
\setlength{\leftskip}{-0.9cm}
$b_{\text{st}}(s)=(s+\frac{11}{12})(s+\frac{13}{12})(s+\frac{15}{12})(s+\frac{16}{12})(s+\frac{17}{12})(s+\frac{19}{12})(s+\frac{20}{12})(s+\frac{21}{12})(s+\frac{23}{12}).$
\end{flushleft}
\noindent 
\begin{tabular}{cccc}
CGS of $\text{Ann}(f^s)$: 630.8 &  \ \ \ & parametric version of Algorithm~4: 341.6 &  \ \ \  \\
\end{tabular}

\item[$\bullet$ $S_{14}$ :] $f=x^2z+yz^{2}+y^5+u_1y^{3}z+u_2z^{3}$ \ \ ($u_1^2-4 \neq 0$)
\begin{flushleft}
{\renewcommand\arraystretch{1.2}
\begin{tabular}{|c|c||c|}
\hline
 $\C^2\backslash \V((u_1^2-4)(3u_1^2-10)u_2)$ & $b_{\text{st}}(s)$ \ \ \ & $\dim_{\C}(H_{M_{(-11/10,f)}})=2$  \\\hline
 $\V(u_2(3u_1^2-10))\backslash \V(u_1^2-4)$ &  $b_{\text{st}}(s)(s+\frac{21}{10})$ & $\dim_{\C}(H_{M_{(-11/10,f)}})=1$ \\
 & & $\dim_{\C}(H_{M_{(-21/10,f)}})=1$ \\ \hline
\end{tabular} 
}
\end{flushleft}
\begin{flushleft}
\setlength{\leftskip}{-0.9cm}
$b_{\text{st}}(s)=(s+\frac{9}{10})(s+\frac{11}{10})(s+\frac{12}{10})(s+\frac{13}{10})(s+\frac{14}{10})(s+\frac{15}{10})(s+\frac{16}{10})(s+\frac{17}{10})(s+\frac{18}{10})(s+\frac{19}{10}).$
\end{flushleft}
\noindent 
\begin{tabular}{cccc}
CGS of $\text{Ann}(f^s)$: 295.1 &  \ \ \ & parametric version of Algorithm~4: 7.438 &  \ \ \  \\
\end{tabular}

\item[$\bullet$ $U_{14}$ : ] $f=x^3+xz^{2}+u_1xy^{3}+y^3z+u_2yz^{3}$ \ \ ($u_1^2+1 \neq 0$) 
\begin{flushleft}
{\renewcommand\arraystretch{1.2}
\begin{tabular}{|c|c||c|}
\hline
 $\C^2\backslash \V((u_1^2+1)(2u_1^2+3)u_2)$ & $b_{\text{st}}(s)$ & $\dim_{\C}(H_{M_{(-10/9,f)}})=2$ \\\hline
$\V(u_2(2u_1^2+3))\backslash \V(u_1^2+1)$ &  $b_{\text{st}}(s)(s+\frac{19}{9})$ & $\dim_{\C}(H_{M_{(-10/9,f)}})=1,$   \\
 & & $\dim_{\C}(H_{M_{(-19/9,f)}})=1$ \\ \hline
\end{tabular} 
}
\end{flushleft}
\begin{flushleft}
\setlength{\leftskip}{-0.9cm}
$b_{\text{st}}(s)=(s+\frac{8}{9})(s+\frac{10}{9})(s+\frac{11}{9})(s+\frac{12}{9})(s+\frac{13}{9})(s+\frac{14}{9})(s+\frac{15}{9})(s+\frac{16}{9})(s+\frac{17}{9}).$
\end{flushleft}
\noindent 
\begin{tabular}{cccc}
CGS of $\text{Ann}(f^s)$: 21230 &  \ \ \ & parametric version of Algorithm~4: 304.3 &  \ \ \  \\
\end{tabular}

\end{enumerate}



\begin{thebibliography}{99}
\bibitem{B89}
J. Brian\c{c}on, M. Granger, P. Maisonobe and M. Miniconi,
\newblock Algorithme de calcul du polyn\^ome de Bernstein : cas non d\'eg\'en\'er\'e. 
\newblock {\it Ann. Inst. Fourier, Grenoble} {\bf 39} (1989), 553--610.

\bibitem{BM}
J. Brian\c{c}on and P. Maisonobe,
\newblock Remarques sur l'id\'eal de Bernstein associ\'e  \`a des polyn\^omes. 
\newblock {\it Preprint}, Univ. Nice Sophia-Antipolis,  no. {\bf 650} (2002).

\bibitem{CN1}
P. Cassou-Nogu\'es,
\newblock Racines de polyn\^omes de Bernstein. 
\newblock {\it Ann. Inst. Fourier, Genoble} {\bf 36} (1986), 1--30.

\bibitem{CN2}
P. Cassou-Nogu\'es,
\newblock Etude du comportement du poly\^ome de Bernstein lors d'une d\'eformation \`a $\mu$ constant de $X^a+Y^b$ avec $(a,b)=1$. 
\newblock {\it Composition Mathematica} {\bf 63} (1987), 291--313.

\bibitem{kw90}
A. Kandri-Rody and V. Weispfenning, 
\newblock Noncommutative Gr\"obner bases in algebras of solvable type.
\newblock \textit{J. Symb. Comp.} {\bf 9} (1990), 1--26.

\bibitem{kashi75}
M. Kashiwara,
\newblock On the maximally overdetermined system of linear differential equations, I, 
\newblock {\it Publications of the Research Institute for Mathematical Sciences} {\bf 10} (1975), 563--579. 


\bibitem{Kashi}
M. Kashiwara,
\newblock $B$-functions and holonomic systems: Rationality of roots of $b$-functions. 
\newblock {\it Invent. Math.}, {\bf 38} (1976), 33--53.

\bibitem{Ka81}
M. Kato, 
\newblock The $b$-function of $\mu$-constant deformation of $x^7 +y^5$, 
\newblock {\it Bull. College of Science, Univ. of the Ryukyus} {\bf 32} (1981), 5--10.

\bibitem{Ka82}
M. Kato, 
\newblock The $b$-function of $\mu$-constant deformation of $x^9 +y^4$, 
\newblock {\it Bull. College of Science, Univ. of the Ryukyus} {\bf 33} (1982), 5--8.

\bibitem{LeR76}
L\^e D\~ung Tr\'ang and C. P. Ramanujam, The invariance of Milnor's number 
implies the invariance of the topological type, {\it Amer. J. Math.} {\bf 98} (1976), 67--78.


\bibitem{LM08}
V. Levandovskyy and J. Mart\'in-Morales,
\newblock Computational $D$-module theory with SINGULAR, comparison with other systems and two new algorithms.
\newblock {\it Proc. ISSAC2008}, ACM, (2008), 173--180.

\bibitem{LM12}
V. Levandovskyy and J. Mart\'in-Morales,
\newblock Algorithms for checking rational roots of $b$-functions and their applications.
\newblock {\it J. Algebra}, {\bf 352} (2012), 408--429.

\bibitem{MN91}
Z. Mebkhout and L. Narv\'aez-Macarro,
\newblock La th\'eorie du polyn\~ome de Bernstein-Sato pour les alg\`ebres de Tate et de Dwork-Monsky-Washintzer.
\newblock {\it Ann. Sci. \'Ecole Norm. Sup.} {\bf (4) 24} no.2 (1991), 227--256.


\bibitem{NOT16} 
K. Nabeshima, K. Ohara and S. Tajima,  
\newblock Comprehensive Gr\"obner systems in rings of differential operators, holonomic $D$-modules and $b$-functions, 
\newblock {\it Proc. ISSAC 2016}, ACM, (2016), 349--356. 

\bibitem{NOT18} 
K. Nabeshima, K. Ohara and S. Tajima,  
\newblock Comprehensive Gr\"obner systems in PBW algebras, Bernstein-Sato ideals and holonomic $D$-modules, 
\newblock \textit{J. Symb. Comp.}, {\bf 89} (2018), 146--170.


\bibitem{NT17} 
K. Nabeshima and S. Tajima, 
\newblock Algebraic local cohomology with parameters and parametric standard bases for zero-dimensional ideals, 
\newblock  \textit{J. Symb. Comp.}, \textbf{82} (2017), 91--122. 

\bibitem{NT17a}
K. Nabeshima and S. Tajima, 
\newblock Comprehensive Gr\"obner systems aprroach to $b$-functions of $\mu$-constant deformations.
\newblock \textit{Saitama Mathematical Journal}, \textbf{31} (2017), 115--136.

\bibitem{NN10}
K. Nishiyama and M. Noro, 
\newblock Stratification associated with local $b$-functions. 
\newblock {\it J. Symb. Comp.}, {\bf 45} (2010), 462--480.


\bibitem{NT92} 
M. Noro and T. Takeshima, 
\newblock  Risa/Asir- A computer algebra system. 
\newblock {\it Proc. ISSAC 1992}, ACM, (1992), 387--396. 
\url{http://www.math.kobe-u.ac.jp/Asir/asir.html}

\bibitem{Oaku1}
T. Oaku,
\newblock An algorithm of computing $b$-functions. 
\newblock {\it Duke Math. J.} {\bf 87} (1997), 115--132. 




\bibitem{TNN09}
S. Tajima, Y. Nakamura and K. Nabeshima,
\newblock Standard bases and algebraic local cohomology for zero dimensional ideals.
\newblock {\it Advanced Studies in Pure Mathematics} \textbf{56} (2009), 341--361.




\bibitem{yano78}
T. Yano,
On the theory of $b$-functions. 
\newblock {\it Pub. Res. Inst. Math. Sci.} {\bf 14} (1978), 111--202.


\bibitem{YS}
E. Yoshinaga and M. Suzuki,  
\newblock  Normal forms of nondegenerate quasihomogeneous functions with inner modality $\le 4$. 
\newblock  {\it Invent. Math.}, \textbf{55} (1979), 185--206. 

\bibitem{Va}
A. N. Varchenko, A lower bound for the codimension of the $ \mu=$const stratum in terms of the mixed Hodge structure, {\it Moscow Univ. Math. Bull.} {\bf 37} (1982), 30--33.

\end{thebibliography}
\end{document}